\documentclass[11pt]{amsart}
\usepackage{graphpap}
\usepackage{amssymb}
\usepackage{hyperref}
\usepackage{xcolor}
\usepackage{color}
\newtheorem{proposition}{Proposition}
\newtheorem{lemme}[proposition]{Lemma}
\newtheorem{theo}[proposition]{Theorem}
\newtheorem{corollaire}[proposition]{Corollary}

\newtheorem{remarque}[proposition]{Remark}

\numberwithin{equation}{section}
\numberwithin{proposition}{section}

\def\11{{\rm 1~\hspace{-1.4ex}l} }
\def\R{\mathbb R}

\def\Z{\mathbb Z}
\def\N{\mathbb N}

\def\T{\mathbb T}

\def\wtilde{\widetilde}
\begin{document}
	\title[Super-critical wave equations ]
	{
		New examples of probabilistic well-posedness for nonlinear wave equations 
	}
	\author{Chenmin Sun, Nikolay Tzvetkov}
	\address{
		Universit\'e de Cergy-Pontoise,  Cergy-Pontoise, F-95000,UMR 8088 du CNRS}
	\email{nikolay.tzvetkov@u-cergy.fr}
	\email{
		chenmin.sun@u-cergy.fr}
	\begin{abstract} 
		We consider fractional wave equations with exponential or arbitrary polynomial nonlinearities.  We  prove the global well-posedness on the support of the corresponding Gibbs measures.
		We provide ill-posedness constructions showing that the results are truly super-critical in the considered functional setting. 
		We also present a result in the case of a general randomisation in the spirit of the work by N.~Burq and the second author. 
	\end{abstract}

	\maketitle
	%
	
	\section{Introduction}
	Our goal in this work is to give new examples of probabilistic well-posedness for nonlinear wave equations with data of super-critical regularity. 
	More precisely, we consider fractional wave equations with exponential or arbitrary polynomial nonlinearities. 
	We  will prove the global well-posedness on the support of the corresponding Gibbs measures (and also a result for more general random initial data). 
	We will also provide ill-posedness constructions showing that the considered problem is 
	super-critical in the sense that the obtained solutions crucially depend on the particular regularisations of the initial data.   
	Let us recall that in the case of a deterministic low regularity well-posedness for dispersive PDE's, the obtained solutions can be seen as limits of  approximated smooth solutions,  independently of the choice of the approximation of the low regularity initial data (see e.g. \cite{CKSTT,IK,KT}). 
	\subsection{The case of exponential nonlinearity}
	Let $(M,g)$ be a compact smooth riemannian manifold of dimension $d$ without boundary.  Let $\Delta_g$ be the associated Laplace-Beltrami operator. For $\sigma\in\R$, we set
	$
	D^\sigma=(1-\Delta_g)^{\sigma/2}.
	$
	Consider the following (fractional) wave equation
	\begin{equation}\label{1}
	\partial_t^2 u+D^{2\alpha} u+e^{u}=0\,,
	\end{equation}
	where $u:\R\times M\rightarrow\R$.
	The case $\alpha=1$ corresponds to the usual wave (or Klein-Gordon) equation posed on $M$. 
	Let $u$ be a smooth solution of \eqref{1}.  If we multiply \eqref{1} by $\partial_t u$ and integrate over $M$, we get
	$$
	\frac{d}{dt}\Big[\int_{M}\big( \frac{1}{2}(\partial_t u)^2+\frac{1}{2}(D^\alpha u)^2+e^u\big)\Big]=0\,.
	$$
	Let $(\varphi_n)_{n\geq 0}$ be an orthonormal basis of $L^2(M)$ of eigenfunctions of $-\Delta_g$ associated with increasing eigenvalues $(\lambda_n^2)_{n\geq0}$.
	By the Weyl asymptotics $\lambda_n\approx n^{\frac{1}{d}}$. With $v = \partial_t u$, we rewrite \eqref{1} as the following first order system:
	\begin{equation}\label{sys}
	\partial_t u=v, \quad \partial_t v=-D^{2\alpha} u-e^u.
	\end{equation}
	The system \eqref{sys} is a Hamiltonian system of PDEs with the Hamiltonian:
	\begin{equation}\label{H}
	H(u,v)=\frac{1}{2}\int_{M}\big((D^\alpha u)^2 + v^2\big) +\int_M e^u.
	\end{equation}
	The Hamiltonian $H(u,v)$ controls the ${\mathcal H^{\alpha}}$ norm of $(u,v)$, where we denote 
	$$
	{\mathcal H^{s}}\equiv H^s(M)\times H^{s-\alpha}(M)
	$$
	for any $s\in\R$,  and $H^s(M)$ is the classical  Sobolev space of order $s$.  For $\alpha\geq \frac{d}{2}$,
\begin{equation}\label{Moser-Trudinger}
	\Big|\int_M e^u\Big|
	\leq C
	e^{C\|u\|_{H^\alpha(M)}^2}
\end{equation}
	and therefore for these values of $\alpha$ the potential energy can be seen as a perturbation. 
	In particular we can show the global well-posedness of \eqref{sys} for data in ${\mathcal H^{\alpha}}$.
	\begin{theo}\label{thm1}
		Let $\alpha\geq \frac{d}{2}$.
		The Cauchy problem associated with \eqref{sys} is (deterministically) globally well-posed for data in ${\mathcal H^{s}}$, $s\geq \alpha$.
	\end{theo}
We omit the proof of Theorem \ref{thm1}. It follows from a standard fix point argument, based on \eqref{Moser-Trudinger}.

	Let $(g_n,h_n)_{n\geq 0}$ be a family of independent standard gaussians on the probability space $(\Omega,\mathcal{F},\mathbb{P})$.  
	The gaussian measure $\mu$ is the image measure under the map $\omega \mapsto (u^\omega,v^\omega)$ defined by
	\begin{equation}\label{random_data}
	u^\omega(x) = \sum_{n =0}^\infty \frac{g_n(\omega)}{\langle \lambda_n\rangle ^{\alpha}}\varphi_n(x)\, , \quad
	v^\omega(x) = \sum_{n = 0}^\infty h_n(\omega)\varphi_n(x)\, .
	\end{equation}
	We can see $\mu$ as a probability measure on ${\mathcal H}^\sigma$, $\sigma<\alpha-d/2$.
	One has the following key property of $\mu$.
	\begin{proposition}\label{ho}
		Let $\alpha>d/2$. For $\theta<\alpha-d/2$, $ \|D^\theta u\|_{L^\infty(M)}$ is finite $\mu$-almost surely. More precisely 
		$$
		\mu\{u:\,  \|D^\theta u\|_{L^\infty(M)}>\lambda\}\leq C\, e^{-c\lambda^2}
		$$
		for some $C, c>0$ independent of $\lambda\geq 1$. 
	\end{proposition}
The proof of Proposition \ref{ho} follows directly from the same argument as in the proof of Proposition \ref{LD}, in particular from \eqref{last}.	Applying Proposition~\ref{ho} with $\theta=0$ we obtain that 
	$
	\int_{M}e^u
	$
	is finite $\mu$ almost surely and we can define the Gibbs measure $\rho$  associated with \eqref{sys} as 
	$$
	d\rho(u,v)=e^{-\int_M e^u}d\mu(u,v)\,.
	$$
	Indeed, using that if $u=\sum_{n} c_n \varphi_n$ then 
	$$
	\frac{1}{2}\int_{M} (D^\alpha u)^2=\sum_{n=0}^\infty \langle \lambda_n\rangle ^{2\alpha}c_n^2,
	$$
	we deduce that one may interpret $\mu$ as a renormalisation of the formal measure 
	$$
	Z^{-1}\, e^{-\frac{1}{2}\int_{M}((D^\alpha u)^2 + v^2)} \,\, du\, dv
	$$
	and therefore $\rho$ becomes $ Z^{-1} e^{-H(u,v)} dudv$ which reminds the Gibbs measure for finite dimensional systems. 
	We have the following result. 
	\begin{theo}\label{thm2}
		Let $\alpha>\frac{d}{2}$ and $0<\sigma<\alpha-\frac{d}{2}$.
		The problem \eqref{sys} is $\mu$ almost surely globally well-posed. Moreover $\rho$ is invariant under the resulting flow $\Phi(t)$ in the following sense. There exists a measurable set $\Sigma\subset\mathcal{H}^{\sigma}$ with full $\mu$ measure, such that $\Phi(t)(\Sigma)=\Sigma$ and for any measurable set $A\subset\Sigma$, we have $\rho(A)=\rho(\Phi(t)A)$ for all $t\in\mathbb{R}$.
	\end{theo}
	The main difficulty in Theorem~\ref{thm2} comes from the fact that the random data \eqref{random_data} is not in the scope of applicability of the deterministic well-posedness result of
	Theorem~\ref{thm1}.  We however have the following connection between Theorem~\ref{thm1} and Theorem~\ref{thm2}.
	\begin{theo}\label{thm3}
		Let $(u_N^\omega(t,x),v_N^\omega(t,x))$ be the solution of \eqref{sys} given by Theorem~\ref{thm1} with smooth data 
		$$
		u^\omega_N(x) = \sum_{\lambda_n\leq N} \frac{g_n(\omega)}{\langle \lambda_n\rangle ^{\alpha}}\varphi_n(x)\, , \quad
		v^\omega_N(x) = \sum_{\lambda_n\leq N} h_n(\omega)\varphi_n(x)\, .
		$$
		Then almost surely in $\omega$,  $(u_N^\omega(t,x),v_N^\omega(t,x))$  converges to the solution 
		$(u,v)$ of \eqref{sys} constructed in Theorem~\ref{thm2}, in $C(\mathbb{R};\mathcal{H}^{\sigma})$ (and $u_N^\omega$ converges to $u$ in
		$L^{\infty}_{loc}(\mathbb{R}\times M)$).
	\end{theo}	
	The restriction $\alpha>d/2$ in Theorem~\ref{thm2} and Theorem~\ref{thm3} is optimal in the sense that for $\alpha=d/2$ the construction of the measure $\rho$ fails because $\int e^u$ is ill defined on the support of $\mu$. However, for $d=2$ one may suitably renormalise $\int e^u$. Such a renormalisation would unfortunately lead to a change of the equation. 
	In the case of the renormalisation used in \cite{GRV,DKRV,DRV} one would obtain the wave equation, without the mass term, but with a source term, related to the curvature of $M$. One can also use a renormalisation as in \cite{BTT-Toulouse} which would avoid the source term in the equation but the mass term should be kept. In the case of both renormalisations we have just mentioned, one can apply  compactness techniques as employed in \cite{BTT-Toulouse, OTh}. The obtained solutions would however be non unique and an approximation result as the one of 
	Theorem~\ref{thm3} is completely out of reach of the scope of the applicability of these weak solution techniques. We believe that obtaining a result as Theorem~\ref{thm3} in the case $\alpha=1$ ($d=2$) for the above mentioned renormalised equations is an interesting and challenging problem.  It is worth mentioning that the relevant parabolic equations  have been  recently studied in \cite{Ga}.
	\\
	
	A novelty in the proof of Theorem~\ref{thm2} and Theorem~\ref{thm3} compared to \cite{B,bourgain,BT-JEMS,BTT-AIF,Tz-Lecture} is that because of the exponential nonlinearity, we need to prove probabilistic Strichartz estimates involving $L^\infty$ norms with respect to the time variables. 
	\subsection{The case of arbitrary polynomial nonlinearity}
	The strategy to prove Theorem ~\ref{thm2} and Theorem~\ref{thm3}  works equally well for power type nonlinearity of an arbitrary degree, as follows
	\begin{equation}\label{wave-polynome}
	\partial_t^2u+D^{2\alpha}u+u^{2k+1}=0.
	\end{equation}
	We have the following statement in the context of \eqref{wave-polynome}.
	\begin{theo}\label{thmpolynomepositive}
		Let $\alpha>\frac{d}{2}$ and $0<\sigma<\alpha-\frac{d}{2}$. Then \eqref{wave-polynome} is $\mu$ almost surely globally well-posed. 
	\end{theo}
	Note that from the scaling consideration:
	$ u(t,x)\mapsto \lambda^{\frac{\alpha}{k}}u(\lambda^{\alpha}t,\lambda x),
	$
	the critical index is $s_c=\frac{d}{2}-\frac{\alpha}{k}.$ Therefore, if  $s<\frac{d}{2}-\frac{\alpha}{k}$, \eqref{wave-polynome} is super-critical with respect to $H^s$. 
	As a consequence, we have an ill-posedness result, Proposition \ref{ill-proposition}, proved in Section~6.
	\\
	
	We underline that Theorem \ref{thmpolynomepositive} is really a super-critical result, in the sense that the way to approximate the solution in $C(\mathbb{R};\mathcal{H}^{\sigma})$ by smooth solutions is very sensitive. More precisely, as a consequence of 
	Theorem~\ref{thmpolynomepositive} and Proposition~\ref{ill-proposition}, we have the following remarkable statement. 
	\begin{corollaire}\label{corollaire1}
		\begin{itemize}
		Assume that $\displaystyle{0<\sigma<\alpha-\frac{d}{2}}$. Let $k\in\N$ such that
			$$\frac{(2k-1)d}{2k+1}<\alpha<\frac{kd}{2},  k\in\mathbb{N}, \,\,\textrm{ and }\,\, \frac{d}{2}-\frac{2\alpha}{2k-1}\leq \sigma<\frac{d}{2}-\frac{\alpha}{k}.$$
			\item For almost every  $(u_0^{\omega},u_1^{\omega})\in {\mathcal H}^{\sigma}$, there exists a sequence 
			$$
			u_N^{\omega}(t,x)\in C(\mathbb{R};C^{\infty}(M)),\quad N=1,2,\cdots
			$$
			of  global solutions  to \eqref{wave-polynome} 
			such that
			$$ \lim_{N\rightarrow \infty}\|(u_{N}^{\omega}(0),\partial_t u_{N}^{\omega}(0))-(u_0^{\omega},u_1^{\omega})\|_{{\mathcal H}^{\sigma}}\rightarrow 0,
			$$
			while for every $T>0$,
			$$ \lim_{N\rightarrow\infty}\|u_n^{\omega}(t)\|_{L^\infty([0,T];H^{\sigma}(M))}\rightarrow\infty.
			$$
			\item 
			Let $(u_N^\omega(t,x),v_N^\omega(t,x))$ be the solution of   \eqref{wave-polynome}  with smooth data 
			$$
			u^\omega_N(x) = \sum_{\lambda_n\leq N} \frac{g_n(\omega)}{\langle \lambda_n\rangle ^{\alpha}}\varphi_n(x)\, , \quad
			v^\omega_N(x) = \sum_{\lambda_n\leq N} h_n(\omega)\varphi_n(x)\, .
			$$
			Then almost surely in $\omega$,  $(u_N^\omega(t,x),v_N^\omega(t,x))$  converges to the solution 
			$(u,v)$ of  \eqref{wave-polynome}  constructed in Theorem~\ref{thmpolynomepositive} in $C(\mathbb{R};\mathcal{H}^{\sigma})$.			
		\end{itemize}
	\end{corollaire}
\begin{remarque}
	For fixed $k\in\mathbb{N}$, if 
	$$
	\frac{(2k-1)d}{2k+1}<\alpha<\frac{kd}{2},
	$$
	then $$\left(\frac{d}{2}-\frac{2\alpha}{2k-1},\frac{d}{2}-\frac{\alpha}{k}\right)\cap \left(0,\alpha-\frac{d}{2}\right)\neq \emptyset.$$
\end{remarque} 
\begin{remarque}
The first assertion of this corollary will follow from the strong ill-posedness result of Proposition \ref{ill-proposition}. Unlike usual ill-posedness construction near the zero initial data, we prove norm-inflation near any smooth data of arbitrary size. The restriction $\alpha>\frac{(2k-1)d}{2k+1}$ here is only a technical assumption (see case 2 in the proof of Lemma \ref{Energyestimate} for detailed discussion). It would be interesting to decide whther the same conclusion holds for the full range $\frac{d}{2}<\alpha<\frac{kd}{2}$. 	
\end{remarque}
	\subsection{General randomisations}
	We remark that for the polynomial nonlinearity, if the underlying manifold is $M=\T^d$, we could also treat the general randomisation introduced in \cite{BT-JEMS}. More precisely, for any  
	$ (u_0,v_0) \in H^s(\T^d)\times H^{s-1}(\T^d),$
	$$
	u_0(x)=a_0+\sum_{n\in \Z^d\setminus\{0\}}(a_{n,1}\cos(n\cdot x)+a_{n,2}\sin(n\cdot x)),$$
	$$ v_0(x)=b_0+\sum_{n\in\Z^d\setminus\{0\}}(b_{n,1}\cos(n\cdot x)x+b_{n,2}\sin(n\cdot x)),
	$$
	we consider the randomisation around $(u_0,v_0)$:
	\begin{equation}\label{randomization-general}
	\begin{split}
	&u_0^{\omega}(x)=a_0g_0(\omega)+\sum_{n\in \Z^d\setminus\{0\}}(a_{n,1}g_{n,1}(\omega)\cos(n\cdot x)+a_{n,2}g_{n,2}(\omega)\sin(n\cdot x)),\\
	& v_0^{\omega}(x)=b_0h_0(\omega)+\sum_{n\in\Z^d\setminus\{0\}}(b_{n,1}h_{n,1}(\omega)\cos(n\cdot x)x+b_{n,2}h_{n,2}(\omega)\sin(n\cdot x)),
	\end{split}
	\end{equation}
	where $\{ g_{0}(\omega), g_{n,j}(\omega), h_0(\omega), h_{n,j}(\omega):n\in \Z^d\setminus\{0\}, j=1,2\}$ are independent standard  Gaussian variables.
	Denote by
\begin{equation*}
\begin{split}
&u_{0,N}^{\omega}(x)=a_0g_0(\omega)+\sum_{|n|\leq N, n\neq 0}(a_{n,1}g_{n,1}(\omega)\cos(n\cdot x)+a_{n,2}g_{n,2}(\omega)\sin(n\cdot x)),\\
& v_{0,N}^{\omega}(x)=b_0h_0(\omega)+\sum_{|n|\leq N, n\neq 0}(b_{n,1}h_{n,1}(\omega)\cos(n\cdot x)x+b_{n,2}h_{n,2}(\omega)\sin(n\cdot x)),
\end{split}
\end{equation*}
	We have the following almost surely global existence as well as uniqueness theorem:
	\begin{theo}\label{thm-general}
		Assume that $M=\T^d$, $\alpha>\frac{d}{2}$. Let $(u_0,v_0)\in \mathcal{H}^s$ with $\frac{(k-1)\alpha}{k}<s<\alpha$. Then almost surely in $\omega\in \Omega$,  \eqref{wave-polynome} with initial data $(u_0^{\omega}, v_0^{\omega})$ is globally well-posed. Moreover, the sequence of smooth solutions $u_N(t)$ to \eqref{wave-polynome} with initial datum $(u_{0,N}^{\omega}, v_{0,N}^{\omega})$ converges in $C(\R;H^s(\T^d)\times H^{s-1}(\T^d))$ to the solution $u^{\omega}(t)$ with initial data $(u_0^{\omega}, v_0^{\omega})$.
	\end{theo}
	We will sketch the proof in  Section~7. 
	The only additional ingredient in the proof of Theorem~\ref{thm-general} is an energy a priori estimate, following the method of Oh-Pocovnicu \cite{OTh-Po} 
	(see also \cite{Sun-Xia}). The crucial fact we use to prove the energy estimate is the almost sure $L^{\infty}$ bound for the linear evolution of the Gaussian random initial data. 
	This is the reason to restrict our consideration to $M=\T^d$ in Theorem~\ref{thm-general}.
	For general randomisations on arbitrary manifold, the almost sure  $L^{\infty}$ bound does not always hold true (see for example \cite{A-Tz}). 
	However, as in  \cite{P-R-T}, using the idea of Burq-Lebeau \cite{BL-AENS}, such an $L^\infty$ bound can be achieved by imposing some assumptions on the variations of the Fourier coefficients of $(u_0,v_0)$. 
	\\
	
	It is worth mentioning that there are many situations when the energy method of Oh-Pocovnicu does not cover the results obtained by exploiting the Gibbs measure. Indeed, for general randomisations,  we need $s\rightarrow 
	\alpha$  as $k\rightarrow \infty$ while for data on the support of the Gibbs measure we need $0<s<\alpha-\frac{d}{2}$, {\it independently}  of $k$. 
	\subsection*{Acknowledgement}
	The authors are supported by the ANR grant ODA (ANR-18-CE40-0020-01). 
	The problem considered in this paper is inspired by a talk of Vincent Vargas at the University of Cergy-Pontoise and by a discussion of the second author with R\'emi Rhodes and Vincent Vargas at ENS Paris. 
	\section{Construction of the Gibbs measure}
\subsection{Notations}	Denote by $\Pi_N$ the sharp spectral projector on $E_N=\mathrm{span}\{\varphi_n:\lambda_n\leq N \}$.  
	Let $\pi_N$ be a smooth projector 
		where 
	$$
	\pi_{N}\big(\sum_{n=0}^\infty c_n \varphi_n \big):=\sum_{n=0}^\infty \psi(\lambda_n/N) c_n \varphi_n,
	$$
	where $\psi\in C^\infty_0(\R)$ $\psi(r)\geq 0$ and $\psi(r)\equiv 1$ if $r\leq 1/2$, $\psi(r)\equiv 0$ if $r>1$. By convention $\pi_{\infty}={\rm Id}$. Clearly, we have
	$$ 
	\pi_N\Pi_N=\pi_N,\quad  \pi_N\Pi_{N/2}=\Pi_{N/2}.
	$$
	We will also use the notations:
	$$ \pi_N^{\perp}:=\mathrm{Id}-\pi_N, \quad \Pi_N^{\perp}:=\mathrm{Id}-\Pi_N.
	$$
\subsection{Definition of Gibbs measure}
	Denote by $\wtilde{\mu}_N$ the distribution of the $E_N\times E_N$ valued random variable
	$$ 
	\omega\mapsto \Big(\sum_{\lambda_n\leq N} \frac{g_n(\omega)}{\langle \lambda_n\rangle ^{\alpha}}\varphi_n(x),\sum_{\lambda_n\leq N}h_n(\omega)\varphi_n(x)\Big).
	$$ 
	Consider the Gaussian  measure $\wtilde{\mu}_N$ induced by this map, which is the probability measure on $\mathbb{R}^{2\dim (E_N)}$ defined by
	$$ d\wtilde{\mu}_N=\prod_{\lambda_n\leq N}\frac{\langle\lambda_n\rangle^{\alpha}}{2\pi}e^{-\frac{(\lambda_n^2+1)^{\alpha}a_n^2}{2}-\frac{b_n^2}{2}}da_ndb_n=\frac{1}{Z_N}e^{-H_0(a_n,b_n)}\prod_{\lambda_n\leq N}da_ndb_n\,,
	$$
	where $H_0$ is the free Hamiltonian
	$$ H_0(u,v)=\frac{1}{2}\int_M(|D^{\alpha}u|^2+|v|^2).
	$$
	Now we define a Gaussian measure on $\mathcal{H}^{\sigma}(M)$$(\sigma<\alpha-d/2)$ be the induced probability measure by the map
	$$ \omega\mapsto \Big(\sum_{n =0}^\infty \frac{g_n(\omega)}{\langle \lambda_n\rangle ^{\alpha}}\varphi_n(x),\sum_{n=0}^\infty h_n(\omega)\varphi_n(x)\Big).
	$$
	The measure $\mu$ can be decomposed into
	$ \mu=\mu^N\otimes\wtilde{\mu}_N
	$
	for all $N$, where $\mu^N$ is the distribution of the random variable on $E_N^{\perp}\times E_N^{\perp}$. Now we define the Gibbs measure $\rho$ by
	$$ d\rho(u)=\exp\Big(-\int_M e^u\Big)d\mu(u).
	$$
	We denote by
	$$ F_N(u)=\int_M e^{\pi_N u},\quad F(u)=\int_M e^u
	$$
	and $$G(u)=\exp\Big(-\int_M e^u\Big),\quad  G_N(u)=\exp\Big(-\int_M e^{\pi_N u}\Big).$$
	To be precise, we firstly define its finite dimensional approximations
	\begin{equation}\label{Gibbs_definition}
	d\wtilde{\rho}_N(u)=G_N(u)d\wtilde{\mu}_N(u),\quad d\rho_N(u)=G_N(u)d\mu(u)=d\mu^N\otimes d\wtilde{\rho}_N.
	\end{equation} 
	The following proposition justifies our definition of $d\rho(u)$.
	\begin{proposition}\label{finite_Gibbs}
		We have the following statements :
		\begin{enumerate}
			\item
			The sequences $(F_N(u))_{N\geq 1}$ and $(G_N(u))_{N\geq 1}$ converge to the limits $F(u), G(u)$ in $L^p(d\mu(u))$, $2\leq p<\infty$,  respectively.  In particular, $G(u)$ exists almost surely with respect to $\mu$.
			\item  $G(u)^{-1}=\exp\left(\int_Me^u\right)$ is almost surely  finite with repsect to $\mu$.
			\item 
			$\lim_{N\rightarrow \infty} \rho_N(A)=\rho(A)$, for every Borel set $A\subset \mathcal{H}^{\sigma}$.
		\end{enumerate}
	\end{proposition}

	\begin{proof}[Proof of Proposition~\ref{Gibbs_definition}]
		(1) As mentioned in Remark 3.8 of \cite{Tz-AIF},  in order to prove that $F_N(u), G_N(u)$ converge to $F(u), G(u)$ in $L^p(d\mu)$, it will be sufficient to show  that
		\begin{itemize}
			\item $\|F_N(u)\|_{L^p(d\mu)}, \|G_N(u)\|_{L^p(d\mu)}$ are uniformly bounded. 
			\item $F_N(u), G_N(u)$ converge to $F(u), G(u)$ in measure.
		\end{itemize}  
		Now we verify the boundeness in $L^p(d\mu)$. Note that $G_N(u)\leq 1$, we only need to check for $F_N(u)$.  We write
		$$
		|F_{N}(u)|\leq 	 \mathrm{vol}(M)e^{\|\pi_N u\|_{L^\infty(M)}}
		\,.
		$$
		Therefore using Proposition~\ref{ho}, we write for $\lambda\geq 1$,
		$$
		\mu(u: |F_N(u)|>\lambda)\leq \mu\big(u: \|\pi_N u\|_{L^\infty(M)}>\frac{\log(\lambda)}{C}\big)\leq C'\, e^{-C(\log(\lambda))^2}
		\leq C_l\lambda^{-l},
		$$
		for every $l\geq 1$.  This proves the uniform in $N$ boundedness of $\|F_N(u)\|_{L^p(d\mu)}$ for every $p<\infty$. Next, 
		we claim that  $F_N(u)$ converges in measure to $F(u)$. Once this is justified, the convergence in measure for $G_N(u)$ would follow automatically since $G_N=e^{-F_N}$.
		For $N_1\geq N_2$, we observe that
	\begin{equation*}
		\begin{split}
			|F_{N_1}(u)-F_{N_2}(u)|  \leq & \int_M |\pi_{N_1}u-\pi_{N_2}u|\cdot e^{|\pi_{N_1}u|+|\pi_{N_2}u|}
			\\
			 \leq  &C \|\pi_{N_1}u-\pi_{N_2}u\|_{L^2(M)}\exp\big(\|\pi_{N_1} u\|_{L^\infty(M)} + \|\pi_{N_2} u\|_{L^\infty(M)}\big)\,.
			 \end{split}
		\end{equation*}
		Therefore by Cauchy-Schwarz, 
		$
		\|F_{N_1}(u)-F_{N_2}(u)\|_{L^1(d\mu)}
		$
		is bounded by
		$$C  \big\| \|\pi_{N_1}u-\pi_{N_2}u\|_{L^2(M)}\big\|_{L^2(d\mu)}
		\|\exp\big(\|\pi_{N_1} u\|_{L^\infty(M)} + \|\pi_{N_2} u\|_{L^\infty(M)}\big)\|_{L^2(d\mu)}\, .	
		$$
		The first factor is clearly  going to zero as $N_1,N_2$ go to infinity. 
		One can show that the second factor is uniformly bounded, exactly as in the proof of the  uniform  boundedness of $\|F_N(u)\|_{L^p(d\mu)}$ .
		Therefore  $(F_N(u))$ is a Cauchy sequence in $L^1(d\mu)$ which implies its convergence in mesure.  This in turn implies the convergence in measure of the sequence $(G_N(u))$ ($G_N$ is a continuous function of $F_N$). 
		\\
		
		(2) To show that $\int_Me^{u}$ is almost surely finite, it will be sufficient to verify that
		$$ \mathbb{E}_{\mu}\left[\int_M e^{u} \right]<\infty.
		$$
		From the proof of (1), 
		$\mathbb{E}_{\mu}[e^{\|\pi_Nu\|_{L^{\infty}(M)}}] 
		$ is uniformly bounded in $N$. Thus we conclude by the dominated convergence.
		\\
		
		(3)  It will be sufficient to check that, for all Borel set $A\subset H^{\sigma}(M)$, we have
		$$ \lim_{N\rightarrow\infty}\int_{H^{\sigma}(M)}\mathbf{1}_{u\in A}|G_N(u)-G(u)|d\mu(u)=0.
		$$
		This is a simple consequence of $L^2(d\mu(u))$ convergence, since 
		$$ \int_{H^{\sigma}(M)}\mathbf{1}_{u\in A}|G_N(u)-G(u)|d\mu(u)\leq \|G_N(u)-G(u)\|_{L^2(d\mu)}\mu(A)^{1/2}\rightarrow 0,
		$$
		as $N\rightarrow\infty$.
		This completes the proof of Proposition~\ref{Gibbs_definition}.
	\end{proof}
	\section{Probabilistic local well posedness}
	\subsection{Deterministic local well-posedness result}Consider the following truncated version of \eqref{sys}
	\begin{equation}\label{sys_N}
	\partial_t u=v, \quad \partial_t v=-D^{2\alpha} u-\pi_{N} (e^{\pi_N u}),
	\end{equation}
	with initial data 
	\begin{equation}\label{uu}
	u|_{t=t_0}=u_0,\quad v|_{t=t_0}=v_0.
	\end{equation}
 Let us next define the free evolution.
	The solution of 
	\begin{equation*}
		\partial_t u=v, \quad \partial_t v=-D^{2\alpha} u,
	\end{equation*}
	subject to initial data
	$$
	u(0,x)=u_0(x), \quad v(0,x)=v_0(x)
	$$
	is given by
	$$
	\bar{S}(t)(u_0, v_0) =(S(t)(u_0,v_0), \partial_t S(t)(u_0,v_0)),
	$$ 
	where
	$$
	S(t)(u_0,v_0) =\cos(tD^\alpha)u_0+D^{-\alpha}\sin(tD^\alpha)v_0
	$$
	and
	$$
	\partial_t S(t)(u_0,v_0) =-D^\alpha\sin(tD^\alpha)u_0+\cos(tD^\alpha)v_0\,.
	$$
	Assume that $\epsilon_0\ll 1$, $1\ll r_0<\infty$ such that $$\frac{d}{r_0}<\epsilon_0<\alpha-\frac{d}{2},
	$$ 
	which is possible since $\alpha>\frac{d}{2}$.  In order to establish the probabilistic local well-posednesss as well as globalising the dynamics, we need some auxillary functional spaces $\mathcal{X}^{s,\beta}, \mathcal{Y}^{\beta}, \mathcal{Z}^{\beta}$, defined via the norms
	$$ 
	\|(u,v)\|_{\mathcal{Y}^{\beta}}:=\sum_{l\in\Z}(1+|l|)^{-\beta}\sup_{l\leq t<l+1}\|S(t)(u,v)\|_{W^{\epsilon_0,r_0}(M)},
	$$
	$$ 
	\|(u,v)\|_{\mathcal{X}^{s,\beta}}:=\|(u,v)\|_{\mathcal{H}^s}+\|(u,v)\|_{\mathcal{Y}^{\beta}}
	$$
	and
	$$ 
	\|(u,v)\|_{\mathcal{Z}^{\beta}}:=\sum_{l\in\Z}(1+|l|)^{-\beta}\sup_{l\leq t\leq l+1}\|S(t)(u,v)\|_{L^{\infty}(M)}.
	$$
	Note that $$\frac{\alpha-s_0}{d}=\frac{1}{2}-\frac{1}{r_0} \textrm{ with } s_0=\alpha-\frac{d}{2}+\frac{d}{r_0}>\epsilon_0,$$
	thus $H^{\alpha}(M)\subset W^{\epsilon_0,r_0}(M)\subset L^{\infty}(M)$ (with continuous inclusions). Moreover, 
 $\mathcal{Y}^{\beta}\subset \mathcal{Z}^{\beta}$. 
 The definition of these weighted in time spaces $\mathcal{Y}^{\beta}, \mathcal{Z}^{\beta}$ is inspired by the work of N.~Burq and the second author (a similar definition appears in \cite{BT-JEMS}). The weight $\beta>1$ will be fixed in the sequel to ensure the $l^1$ summation, without other importance. These norms are only designed to treat the linear evolution part of the solution, since unlike \cite{BTT-AIF}, the linear evolution is not periodic in time.

	Denote by $\Phi_N(t)$ the flow of the truncated equation \eqref{sys_N}-\eqref{uu}. In components, we write 
 $$\Phi_N(t)(u_{0},v_{0})=(\Phi^1_N(t)(u_0,v_0),\Phi^2_N(t)(u_0,v_0))$$ with $\Phi_N^1(t)(u_0,v_0)=u_N(t), \Phi^2_N(t)(u_0,v_0)=\partial_tu_N(t)$. Similarly, we denote by
 $$\Phi(t)(u_0,v_0)=(\Phi^1(t)(u_0,v_0), \Phi^2(t)(u_0,v_0)  ),$$ where $\Phi^1(t)(u_0,v_0)=u(t), \Phi^2(t)(u_0,v_0)=\partial_tu(t)$ for solutions of the non truncated equation 
 $$  \partial_t^2u+D^{2\alpha}u+e^u=0, \quad (u,\partial_tu)|_{t=0}=(u_0,v_0).
 $$
 We also denote the nonlinear evolution part by
	$$ W(t)(u_0,v_0):=\Phi^1(t)(u_0,v_0)-S(t)(u_0,v_0),$$
	$$
	W_N(t)(u_{0,N},v_{0,N}):=\Phi_N^1(t)(u_{0,N},v_{0,N})-S(t)(u_{0,N},v_{0,N}).
	$$

	The next proposition contains the local theory for \eqref{sys_N} and the original system \eqref{sys}.
	\begin{proposition}\label{local_theory}
		Let $N \in \N \cup \{\infty\}$ and $\beta>1$. There exist $c>0$ and $\kappa>0$ such that the following holds true.
		The Cauchy problems  \eqref{sys_N}-\eqref{uu} and \eqref{sys}  are locally well-posed for data $(u_0,v_0)$ such that 
		$S(t)(u_0,v_0)\in L^\infty(M)$.  More precisely for every $R\geq 1$ if $(u_0,v_0)$ satisfies 
		\begin{equation}\label{bound}
		\|(u_0,v_0)\|_{\mathcal{Z}^{\beta}}<R,
		\end{equation}
		then there is a unique solution of  \eqref{sys_N}-\eqref{uu}  on $[t_0-\tau,t_0+\tau]$, where 
		\begin{equation}\label{time}
		\tau=ce^{-\kappa R}
		\end{equation}
		which can be written as 
		$$
		(u,v)=\bar{S}(t-t_0)(u_0,v_0)+(\tilde{u},\tilde{v}),
		$$
		with 
		$$
		\|(\tilde{u},\tilde{v})\|_{{\mathcal H^{\alpha}}}\leq C\,.
		$$
		In particular, the Cauchy problems \eqref{sys_N}-\eqref{uu} and \eqref{sys}  are locally well-posed for data $(u_0,v_0)\in \mathcal{X}^{\sigma,\beta}$, for all $t\in[t_0-\tau,t_0+\tau]$, provided that $\sigma<\alpha-\frac{d}{2}$.
	\end{proposition}
	\begin{proof}
		We can write  \eqref{sys_N}-\eqref{uu}  as
		\begin{equation*}
			(u(t),v(t))=\bar{S}(t-t_0)(u_0,v_0)+\big(F_1(u),F_2(u)\big),
		\end{equation*}
		where
		\begin{align*}
			F_1(u)& =-\int_{t_0}^t D^{-\alpha}\sin((t-\tau)D^\alpha)\, \pi_N\big(e^{\pi_N u(\tau)} \big)d\tau,\\
			F_2(u)& =-\int_{t_0}^t \cos((t-\tau)D^\alpha)\, \pi_N\big( e^{\pi_N u(\tau)}\big) d\tau.
		\end{align*}
		If we write $(u,v)=\bar{S}(t-t_0)(u_0,v_0)+(\tilde{u},\tilde{v})$, we obtain that $(\tilde{u},\tilde{v})$ solves 
		$$
		\partial_t \tilde{u}=\tilde{v}, \quad \partial_t \tilde{v}=-D^{2\alpha} \tilde{u}-\pi_{N} (e^{\pi_N (\tilde{u}+S(t)(u_0,v_0))})
		$$
		with zero initial data.  Therefore $\tilde{u}$ solves  
		\begin{equation}\label{int_eq}
		\tilde{u}(t)=F_1(\tilde{u}+S(t-t_0)(u_0,v_0) ) \,.
		\end{equation}
		Once we solve \eqref{int_eq}, we recover $\tilde{v}$ by  $\tilde{v}= \partial_t \tilde{u}$.
		Define the map $\Phi_{u_0,v_0}$ by
		$$
		\Phi_{u_0,v_0}(u):=F_1(u+S(t-t_0)(u_0,v_0) )\,.
		$$
		It follows from the definition that 
		\begin{equation}\label{bgt}
		\begin{split}
			\|\Phi_{u_0,v_0}(u)\| _{L^\infty([t_0, t_0+T]; H^\alpha (M))} \leq  &CT\|e^{\pi_N (u+S(t-t_0)(u_0,v_0))}\|_{L^\infty(I;L^2(M))}
			\\
			\leq 
			&CT({\rm vol}(M))^{\frac{1}{2}}\|e^{\pi_N (u+S(t-t_0)(u_0,v_0))}\|_{L^\infty(I;L^\infty(M))}\,
		\end{split}
		\end{equation}
			where $I=[t_0,t_0+T]$. Since $\pi_N$ is bounded on $L^\infty$ (see e.g. \cite{BGT}), a use of the Sobolev embedding $H^\alpha(M)\subset L^\infty(M)$  yields the estimate 
		$$
		\|\Phi_{u_0,v_0}(u)\| _{L^\infty([I;H^\alpha (M))} \leq CT e^{C\|S(t-t_0)(u_0,v_0))\|_{L^\infty(I\times M)}}e^{\|u\|_{L^\infty(I;H^\alpha(M))}}\,.
		$$
	  Note that
		$$ \|S(t-t_0)(u_0,v_0)\|_{L^{\infty}(I\times M)}\leq C\|(u_0,v_0)\|_{\mathcal{Z}^{\beta}},
		$$
		hence under \eqref{bound} we deduce that 
		\begin{equation}\label{contr1}
		\|\Phi_{u_0,v_0}(u)\| _{L^\infty(I;H^\alpha (M))} \leq CT e^{CR}e^{\|u\|_{L^\infty(I;H^\alpha(M))}}\,.
		\end{equation}
		Similarly we obtains that
		\begin{equation}\label{contr2}
		\begin{split}
		&	\|\Phi_{u_0,v_0}(u)-\Phi_{u_0,v_0}(v)\| _{L^\infty(I;H^\alpha (M))} 
			\\
			\leq &CT e^{CR}\exp\big(\|u\|_{L^\infty(I;H^\alpha(M))}+\|v\|_{L^\infty(I;H^\alpha(M))}\big)\|u-v\|_{L^\infty(I;H^\alpha(M))}\,.
		\end{split}
		\end{equation}
		Define the space $X_T$ as
		$$
		X_T=\{u\in C(I;H^\alpha(M))\,:\, \|u\|_{L^\infty(I;H^\alpha(M))}\leq 1\}.
		$$
		Using \eqref{contr1}, we obtain that for $T$ as in \eqref{time} the map $\Phi_{u_0,v_0}$ enjoys the property 
		$$
		\Phi_{u_0,v_0}(X_T)\subset X_T.
		$$ 
		Under the same restriction on $T$, thanks to \eqref{contr2}, we obtain that the map  $\Phi_{u_0,v_0}$ is a contraction on $X_T$. 
		The fixed point of this contraction is the solution of \eqref{int_eq}. 
		This completes the proof of Proposition~\ref{local_theory}.
	\end{proof}
In the same spirit of the proof, we establish a local convergence result, which will be needed to construct global dynamics in Section~4.	
		\begin{lemme}\label{localconvergence}
		Assume that $0<\sigma<\alpha-\frac{d}{2}$ and $\beta>1$. There exist $R_0>0, c>0, \kappa>0$ such that the following holds true. Consider a sequence $(u_{0,N_p},v_{0,N_p})\in E_{N_p}\times E_{N_p}$ and $(u_0,v_0)\in \mathcal{X}^{\sigma}$, with $N_p\rightarrow \infty$. Assume that there exists $R>R_0$ such that
		$$\|(u_{0,N_p},v_{0,N_p})\|_{\mathcal{X}^{\sigma,\beta}}\leq R,\quad \|(u_0,v_0)\|_{\mathcal{X}^{\sigma,\beta}}\leq R, $$ and
		$$ \lim_{p\rightarrow\infty}\|\pi_{N_p}(u_{0,N_p},v_{0,N_p})-(u_0,v_0)\|_{\mathcal{H}^{\sigma}}=0.
		$$ 
		Then if we set $\tau=ce^{-\kappa R}$, the flow $\Phi_{N_p}(t)(u_{0,N_0},v_{0,N_p}), \Phi(t)(u_0,v_0)$ exist for $t\in[-\tau,\tau]$ and satisfy
		$$ \|\Phi_{N_p}(t)(u_{0,N_p},v_{0,N_p})\|_{L^{\infty}([-\tau,\tau];\mathcal{X}^{\sigma,\beta})}\leq R+1, \quad \|\Phi(t)(u_0,v_{0,})\|_{L^{\infty}([-\tau,\tau];\mathcal{X}^{\sigma,\beta})}\leq R+1.
		$$
		Furthermore,
		$$ \lim_{p\rightarrow\infty}\|\pi_{N_p}\Phi_{N_p}(t)(u_{0,N_p},v_{0,N_p})-\Phi(t)(u_0,v_0)\|_{L^{\infty}([-\tau,\tau];\mathcal{H}^{\sigma})}=0,
		$$
		and
		$$ \lim_{p\rightarrow\infty}\|\pi_{N_p}W_{N_p}(t)(u_{0,N_p},v_{0,N_p})-W(t)(u_0,v_0)\|_{L^{\infty}([-\tau,\tau];H^{\alpha}(M))}=0.
		$$
	\end{lemme}
	
	\begin{proof}
		The existence of $\Phi_{N_p}(t)(u_{0,N_p},v_{0,N_p})$ and $\Phi(t)(u_0,v_0)$ as well as the bound on $[-\tau,\tau]$ are guaranteed by the local well-posedness result, Proposition~\ref{local_theory}. We only need to prove the convergence.
		
		Denote by
		$$ (u_p(t),\partial_tu_p(t))=\Phi_{N_p}(t)(u_{0,N_p},v_{0,N_p}),\quad (u(t),\partial_t u(t))=\Phi(t)(u_0,v_0).
		$$
		From local theory, we can write
		$$ u_p(t)=S(t)(u_{0,N_p},v_{0,N_p})+w_p(t),\quad u(t)=S(t)(u_0,v_0)+w(t),
		$$
		such that 
		$$ \|(w_p,\partial_tw_p)\|_{L^{\infty}([0,\tau];\mathcal{H}^{\alpha})}\leq R+1, \quad \|(w,\partial_tw)\|_{L^{\infty}([0,\tau];\mathcal{H}^{\alpha})}\leq R+1.
		$$ 
		The convergence of the linear part
		$$ \|\pi_{N_p}\overline{S}(t)(u_{0,N_p},v_{0,N_p})-\overline{S}(t)(u_0,v_0)\|_{L^{\infty}([0,\tau];\mathcal{H}^{\alpha})}=0
		$$
		follows from the assumption and the boundedness of $\pi_{N_p}$ on $W^{\epsilon_0,r_0}(M)$.
		
		Next we estimate the nonlinear part 
		$$ \|\pi_{N_{p}}w_p(t)-w(t)\|_{H^{\alpha}(M)}.
		$$
		Writing $w_p,w$ by the Duhamel formula and using triangle inequality,  we can bound the quantity above by the three contributions:
		\begin{equation*}
			\begin{split}
				&\mathrm{I}(t):=\int_0^t\left\|e^{\pi_{N_p}(w_p(t')+S(t')(u_{0,N_p},v_{0,N_p}))}-e^{\pi_{N_p}(w(t')+S(t')(u_0,v_{0}))}\right\|_{L^2(M)}dt',\\
				&\mathrm{II}(t):=\int_0^t\left\|e^{w(t')+S(t')(u_0,v_0)}-e^{\pi_{N_p}(w(t')+S(t')(u_0,v_0))}\right\|_{L^2(M)}dt',\\
				&\mathrm{III}(t):=\int_0^t\left\|\pi_N^{\perp}e^{w(t')+S(t')(u_0,v_0)}\right\|_{L^2(M)}dt'.
			\end{split}
		\end{equation*}
		Since $$e^{w(t)+S(t)(u_0,v_0)}\in L^{\infty}([-\tau,\tau]\times M),$$  we have that $$\|\mathrm{III}\|_{L^{\infty}([-\tau,\tau])}=o_{p\rightarrow\infty}(1).$$
		For $\mathrm{I}(t)$ and $\mathrm{II}(t)$, we can bound them by
		$$ C(R,\tau)\left(o_{p\rightarrow\infty}(1)+\int_0^t\|\pi_{N_p}w_p(t')-w(t')\|_{L^2(M)}dt'\right).
		$$
		By applying Gronwall inequality, the proof of Lemma \ref{localconvergence} is complete.
	\end{proof}

	\subsection{Large deviation estimate for linear evolutions }
	Proposition~\ref{local_theory} is deterministic. The probabilistic part of the analysis comes from the following statement. 
	\begin{proposition}\label{LD}
		Assume that $\beta>1$.	There are positive constants $C$ and $c$ such that for every $R\geq 1$,
		\begin{equation*}
			\mu\big((u_0,v_0)\,:\,  \|(u_0,v_0)\|_{\mathcal{Y}^{\beta}}\geq R\big)\leq C\, e^{-cR^2}.
		\end{equation*}
		As a consequence, we also have a similar bound for the Gibbs measure 
		\begin{equation*}
			\rho\big((u_0,v_0)\,:\,  \|(u_0,v_0)\|_{\mathcal{Y}^{\beta}}\geq R\big)\leq C\, e^{-cR^2}.
		\end{equation*}
	\end{proposition}
	
	Observe that in the case of an exponential nonlinearity, we need a large deviation estimate for $L^\infty$ norms in time (in \cite{BT-JEMS} only $L^p$ in time for a finite $p$ were established).
	\begin{proof}[Proof of Proposition~\ref{LD}]
		Let $\eta\in C_0^\infty$ be a bump function localising in the interval $[-2,2]$. Denote by $\eta_l(t):=\eta(t-l)$ for $l\in Z$.
		We need to show 
		$$
		\Big\|\sum_{l\in Z}\langle l\rangle^{-\beta}\|\eta_l(t)S(t)(u_0,v_0)\|_{L^{\infty}(\R; W^{\epsilon_0,r_0}(M)}\Big\|_{L^p(d\mu)}\leq C\sqrt{p}\,.
		$$
		Coming back to the definition of $S(t)$, we observe that it suffices to prove the bound 
		$$
		\Big\|\sum_{l\in \Z}\langle l\rangle^{-\beta}\Big\|
		\eta_l(t)\sum_{n=0}^\infty
		\frac{c_n g_n(\omega)}{\langle\lambda_n\rangle^{\alpha}}e^{it\lambda_n^\alpha}\varphi_n(x)
		\Big\|_{L^\infty(\R;W^{\epsilon_0,r_0}(M)
		}\Big\|_{L^p(\Omega)}\leq C\sqrt{p}\,,
		$$
		where $|c_n|$ is bounded. Thanks to $\beta>1$, we have $\sum_{l}\langle l\rangle^{-\beta}<\infty$, thus it suffices to prove the bound
		$$ \left\|\left\|\eta_{l}(t)\sum_{n=0}^{\infty}\frac{c_ng_n(\omega)}{\langle\lambda_n\rangle^{\alpha}}e^{it\lambda_n^{\alpha}}\varphi_n(x)\right\|_{L^{\infty}(\R;W^{\epsilon_0,r_0}(M))}\right\|_{L^p(\Omega)}\leq C\sqrt{p},
		$$
		uniformly in $l\in\Z$.
		
		To simplify the notation, we only write down the case where $l=0$. For other $l$, the arguments are exactly the same.  Using the Sobolev embedding with a large $q$ and a small $\delta$ (depending on $q$) we obtain that its is sufficient to obtain 
		$$
		\Big\|\Big\|
		D_t^{\delta}D^{\epsilon_0}
		\big(\eta(t)\sum_{n=0}^\infty\frac{c_n g_n(\omega)}{\langle\lambda_n\rangle^{\alpha}}e^{it\lambda_n^\alpha}\varphi_n(x)\big)
		\Big\|_{L^q(\R;L^{r_0}（(M)
		}\Big\|_{L^p(\Omega)}\leq C\sqrt{p}\,,
		$$
		where $D_t^\delta=(1-\partial_t^2)^{\delta/2}$. 
		Since $D^{\epsilon_0}\varphi_n=\langle\lambda_n\rangle^{\epsilon_0}\varphi_n$, we are reduced to prove the bound 
		$$
		\Big\|\Big\|
		\sum_{n=0}^\infty\frac{c_n g_n(\omega)}{\langle\lambda_n\rangle^{\alpha-{\epsilon_0}}}
		\alpha_n(t)\varphi_n(x)\Big\|_{L^q(\R;L^{r_0}( M)
		}\Big\|_{L^p(\Omega)}\leq C\sqrt{p}\,,
		$$
		where 
		$
		\alpha_n(t)=D_t^\delta(\eta(t)\exp(it\lambda_n^\alpha)).
		$
		Applying the Minkowski inequality, we deduce that it suffices to prove that for $p\geq q\geq r_0$,
		$$
		\Big\|\Big\|
		\sum_{n=0}^\infty\frac{c_n g_n(\omega)}{\langle\lambda_n\rangle^{\alpha-\epsilon_0}}
		\alpha_n(t)\varphi_n(x)\Big\|_{L^p(\Omega)}
		\Big\|_{L^{r_0}(M; L^{q}(\R))}\leq C\sqrt{p}\,.
		$$
		Observe that in this discussion $q$ is large but fixed and $p$ goes to $\infty$. 
		Now for a fixed $(t,x)$, we can apply the Khinchin inequality and write 
		$$
		\Big\|
		\sum_{n=0}^\infty\frac{c_n g_n(\omega)}{\langle\lambda_n\rangle^{\alpha-\epsilon_0}}
		\alpha_n(t)\varphi_n(x)\Big\|_{L^p(\Omega)}
		\leq 
		C\sqrt{p}
		\Big(
		\sum_{n=0}^\infty
		\frac{|\alpha_n(t)|^2 \, |\varphi_n(x)|^2}{\langle \lambda_n\rangle^{2(\alpha-\epsilon_0)}}
		\Big)^{\frac{1}{2}}\,.
		$$
		Therefore, we reduce the matters to the deterministic bound 
		$$
		\Big\|
		\sum_{n=0}^\infty
		\frac{|\alpha_n(t)|^2 \, |\varphi_n(x)|^2}{\langle \lambda_n\rangle^{2(\alpha-\epsilon_0)}}
		\Big\|_{L^{r_0/2}(M; L^{q/2} (\R)}\leq C\,.
		$$
		For a fixed $x$, we apply the triangle inequality to obtain that 
		\begin{equation}\label{back}
		\Big\|
		\sum_{n=0}^\infty
		\frac{|\alpha_n(t)|^2 \, |\varphi_n(x)|^2}{\langle \lambda_n\rangle^{2(\alpha-\epsilon_0)}}
		\Big\|_{L^{q/2}(\R)}\leq 
		\sum_{n=0}^\infty
		\frac{ |\varphi_n(x)|^2}{\langle \lambda_n\rangle^{2(\alpha-\epsilon_0)}}
		\|\alpha_n\|_{L^q(\R)}^2\,.
		\end{equation}
	
		\begin{lemme}\label{fd}
			There is $C$ such that for every $n\geq 1$,
			$
			\|\alpha_n\|_{L^q(\R)}\leq C\lambda_n^{\alpha\delta}\,.
			$
		\end{lemme}
		
		\begin{proof}
			Let $\beta(t)=\eta(t)e^{it\lambda_n^{\alpha}}$, we have that
			$$ \|\alpha_n\|_{L^q(\mathbb{R})}\leq C\|\beta\|_{L^q(\mathbb{R})}+ C\|\beta\|_{\dot{B}_{q,2}^{\delta}(\mathbb{R})}.
			$$
			From the characterisation of Besov spaces (see \cite{BCT-livre}), we have that
			$$ \|\beta\|_{\dot{B}_{q,2}^{\delta}}^2\sim  \int_{\mathbb{R}}\left(\int_{\mathbb{R}}|\beta(t+\tau)-\beta(t)|^qdt\right)^{\frac{2}{q}}\frac{d\tau}{|\tau|^{1+2\delta}}.
			$$
			We write 
			$$
			\eta(t)\exp(it\lambda_n^\alpha)-\eta(\tau)\exp(i\tau\lambda_n^\alpha)=\exp(it\lambda_n^\alpha)(\eta(t)-\eta(\tau))+\eta(\tau)(\exp(it\lambda_n^\alpha)-\exp(i\tau\lambda_n^\alpha))
			$$
			which yields two contributions. The first one is again uniformly bounded. Therefore the issue is to check that
			$$ \int_{\mathbb{R}}\left(\int_{\mathbb{R}}|\eta(t)(e^{i(t+\tau)\lambda_n^{\alpha}}-e^{it\lambda_n^{\alpha}})|^qdt\right)^{\frac{2}{q}}\frac{d\tau}{|\tau|^{1+2\delta}}\leq (C\lambda_n^{\alpha\delta})^2.
			$$
			For $|\tau|\leq c\lambda_n^{-\alpha}$, we use
			$$  |e^{i\tau\lambda_n^{\alpha}}-1|\leq |\tau|\lambda_n^{\alpha},
			$$
			and thus
			\begin{equation*}
				\begin{split}
					&\int_{|\tau|<c\lambda_n^{-\alpha}}\frac{1}{|\tau|^{1+2\delta}}\left(\int_{\mathbb{R}}|\eta(t)(e^{i(t+\tau)\lambda_n^{\alpha}}-e^{it\lambda_n^{\alpha}})|^qdt\right)^{\frac{2}{q}}d\tau \\ \leq &\int_{|\tau|<c\lambda_n^{-\alpha}}\frac{1}{|\tau|^{1+2\delta}}\left(\int_{\R}|\eta(t)|^q|\tau|^q\lambda_n^{\alpha q}\right)^{\frac{2}{q}}d\tau\leq C\lambda_n^{2\alpha\delta}.
				\end{split}
			\end{equation*}
			The other contribution for $|\tau|>c\lambda_n^{-\alpha}$ can be bounded by
			$$
			\int_{|\tau|>c\lambda_n^{-\alpha}}\left(\int_{\R}|2\eta(t)|^qdt\right)^{\frac{2}{q}}d\tau\leq C\lambda_n^{2\alpha\delta}.
			$$
			This completes the proof of Lemma~\ref{fd}.
			
		\end{proof}
		
		Coming back to \eqref{back} and using Lemma~\ref{fd}, we deduce that it suffices to majorize
		$$
		\Big\|\sum_{n=0}^\infty
		\frac{ |\varphi_n(x)|^2}{\langle \lambda_n\rangle^{2(\alpha-\epsilon_0-\alpha\delta)}}
		\Big\|_{L^{r_0/2}(M)}\,.
		$$
		Using the compactness of $M$, it suffices to get the following estimate 
		\begin{equation*}
			\sup_{x\in M}\,\,
			\sum_{n=0}^\infty\frac{ |\varphi_n(x)|^2}{\langle \lambda_n\rangle^{2(\alpha-\epsilon_0-\alpha\delta)}}
			<\infty.
		\end{equation*}
		Fix $\delta$ sufficiently small such that $\beta:=2(\alpha-\epsilon_0-\alpha\delta)>d$ (here we fix the value of $q$ as well). 
		Therefore we need to show that 
		\begin{equation}\label{last}
		\sup_{x\in M}\,
		\sum_{n=0}^\infty\frac{ |\varphi_n(x)|^2}{\langle \lambda_n\rangle^{\beta}}
		<\infty,\quad \beta>d\,.
		\end{equation}
		Estimate \eqref{last} is direct if $|\varphi_n(x)|$ are uniformly bounded (this is the case of the torus).
		In the case of a general manifold, it is not true that $|\varphi_n(x)|$ are uniformly bounded. However \eqref{last} is true thanks to \cite{H}. 
		More precisely for a dyadic $N$, we can write
		$$
		\sum_{N\leq  \langle \lambda_n\rangle \leq 2N}\frac{ |\varphi_n(x)|^2}{\langle \lambda_n\rangle^{\beta}}
		\leq 
		CN^{-\beta} 
		\sum_{N\leq  \langle \lambda_n\rangle \leq 2N}  |\varphi_n(x)|^2\,.
		$$
		Thanks to \cite{H} there is $C$ such that for every dyadic $N\geq 1$ and every $x\in M$,
		\begin{equation}\label{Hormander}
		\sum_{N\leq  \langle \lambda_n\rangle \leq 2N}  |\varphi_n(x)|^2\leq C N^d.
		\end{equation}

		This readily implies \eqref{last}. The proof of Proposition~\ref{LD} is completed. 
	\end{proof}

We complete this section by proving following probabilistic bound for the tails of sharp spectral truncation, which will be used in the proof of Theorem \ref{thm3}.

	\begin{proposition}\label{probabilistic2}
		Assume that $2\leq r_0<\infty$, $\epsilon_0<\alpha-\frac{d}{2}$, $0<s<\alpha-\frac{d}{2}-\epsilon_0$ and $\beta>1$. Then there exist $C>0, c>0,$  such that for any $R>0$, we have 
		$$ \mu\left((u_0,v_0):N^s\|\Pi_N^{\perp}(u_0,v_0)\|_{\mathcal{Y}^{\beta}}>R\right)\leq Ce^{-cR^2}.
		$$
	\end{proposition}

	\begin{proof}
		Following the same notations as in the proof of the Proposition~\ref{LD}. It suffices to prove that 
		\begin{equation}\label{probabilistic2-1}
		\left\| \left\|\eta(t)\sum_{\lambda_n\geq N}\frac{c_ng_n(\omega)}{\langle\lambda_n\rangle^{\alpha}}e^{it\lambda_n^{\alpha}}\varphi_n(x)\right\|_{L^{\infty}(\mathbb{R};W^{\epsilon_0,r_0}(M))} \right\|_{L^p(\Omega)}\leq CN^{-s}\sqrt{p},
		\end{equation}
		for $p$ large enough. 
		From Sobolev embedding, we are reduced to prove the bound
		\begin{equation}\label{probabilistic2-2}
		\left\|\left\|\sum_{\lambda_n\geq N}\frac{c_ng_n(\omega)}{\langle\lambda_n\rangle^{\alpha-\epsilon_0}}\alpha_n(t)\varphi_n(x)\right\|_{L^q(\mathbb{R};L^{r_0}(M))}\right\|_{L^p(\Omega)}\leq CN^{-s}\sqrt{p},
		\end{equation}
		where $\alpha_n(t)=D_t^{\delta}(\eta(t)e^{it\lambda_n^{\alpha}})$.
		From Minkowski inequality, it will be sufficient to prove that for $p\geq q,p\geq r_0$,
		\begin{equation}\label{probabilistic2-3}
		\left\|\left\|\sum_{\lambda_n\geq N}\frac{c_ng_n(\omega)}{\langle\lambda_n\rangle^{\alpha-\epsilon_0}}\alpha_n(t)\varphi_n(x)\right\|_{L^p(\Omega)}\right\|_{L^q(\mathbb{R};L^{r_0}(M))}\leq CN^{-s}\sqrt{p}
		\end{equation}
		We can apply the Khinchin inequality and write 
		$$
		\Big\|
		\sum_{\lambda_n\geq N}\frac{c_n g_n(\omega)}{\langle\lambda_n\rangle^{\alpha-\epsilon_0}}
		\alpha_n(t)\varphi_n(x)\Big\|_{L^p(\Omega)}
		\leq 
		C\sqrt{p}
		\Big(
		\sum_{\lambda_n\geq N}
		\frac{|\alpha_n(t)|^2 \, |\varphi_n(x)|^2}{\langle \lambda_n\rangle^{2(\alpha-\epsilon_0)}}
		\Big)^{\frac{1}{2}}\,.
		$$
		By taking $\delta>0$ small and $q\geq r_0$ to be large enough and using Minkowski inequality again, we are reduced to prove the bound
		$$ \left\|\sum_{\lambda_n\geq N}\frac{|\alpha_n(t)|^2|\varphi_n(x)|^2}{\langle\lambda_n\rangle^{2(\alpha-\epsilon_0)}}\right\|_{L^{r_0/2}(M;L^{q/2}(\mathbb{R}))}\leq CN^{-s}.
		$$ 
		Using the bound $\|\alpha_n(t)\|_{L^q(t)}\leq C\lambda_n^{\alpha\delta}$ and \eqref{Hormander}, we have for each $x\in M$,
		\begin{equation*}
			\begin{split}
				\sum_{\lambda_n\geq N}\frac{|\varphi_n(x)|^2}{\langle\lambda_n\rangle^{2(\alpha-\epsilon_0-\alpha\delta)}}=&\sum_{k=0}^{\infty}\sum_{2^kN\leq \lambda_n\leq 2^{k+1}N}\frac{|\varphi_n(x)|^2}{\langle\lambda_n\rangle^{2(\alpha-\epsilon_0-\alpha\delta)}}\\
				\leq &C\sum_{k=0}^{\infty}\frac{(2^kN)^d}{(2^kN)^{2(\alpha-\epsilon_0-\alpha\delta)}}\\
				\leq &C'N^{-2(\alpha-\epsilon_0-\alpha\delta-d/2)},
			\end{split}
		\end{equation*}
		provided that $\alpha-\epsilon_0-\alpha\delta>\frac{d}{2}$. This completes the proof of Proposition~\ref{probabilistic2}.
	\end{proof}

	
	
	\section{Global existence and measure invariance}

	\subsection{Hamiltonian structure for the truncated equation}
	
	We consider here the truncated problem
	\begin{equation}\label{finite}
	\partial_tu=v, \partial_tv=-D^{2\alpha}u-\pi_Ne^{\pi_Nu};\quad (u,v)|_{t=0}=(u_0,v_0)=\pi_N(u_0,v_0)\in E_{N}\times E_{N}.
	\end{equation}
	
	For $(u,v)\in E_{N}\times E_{N}$, we write
	$$ (u,v)=\left(\sum_{\lambda_n\leq N}\psi\left(\frac{\lambda_n}{N}\right)a_n\varphi_n, \sum_{\lambda_n\leq N}\psi\left(\frac{\lambda_n}{N}\right)b_n\varphi_n\right),\quad a_n,b_n\in\mathbb{R}.
	$$
	Consider the Hamiltonian 
	$$ J(a_1,\cdots,a_N;b_1,\cdots;b_N):=\frac{1}{2}\sum_{\lambda_n\leq N}\psi\left(\frac{\lambda_n}{N}\right)^2(\langle\lambda_n\rangle^{2\alpha}a_n^2+b_n^2)+\int_M e^{\sum_{\lambda_n\leq N}\psi\left(\frac{\lambda_n}{N}\right)a_n\varphi_n(x)}.
	$$
	One easily verifies that
	\eqref{finite} is just the Hamiltonian ODE
	$$ \frac{da_n}{dt}=\frac{\partial J}{\partial b_n},\quad \frac{db_n}{dt}=-\frac{\partial J}{\partial a_n}.
	$$
	Recall that  $\Phi_N(t)$ denotes the flow map associated with \eqref{finite}. Thus from Liouville theorem, the measure $\mu_N$ is invariant under the flow $\Phi_N(t)$.

	\subsection{Construction of global solution}

	\begin{proposition}\label{almostglobal}
		Fix $0<\sigma<\alpha-\frac{d}{2}$ and $\beta>1$. There exists a constant $C>0$ such that for all $m\in\mathbb{N}$, there exists a $\rho_N$ measurable set $\wtilde{\Sigma}_N^{m}\subset E_N\times E_N$ so that
		$$ \rho_N(E_N\setminus\wtilde{\Sigma}_N^m)\leq 2^{-m}.
		$$
		For all $(f,g)\in\wtilde{\Sigma}_N^m$, $t\in\R$, 
		\begin{equation}\label{bd1}
		\|\Phi_N(t)(f,g)\|_{\mathcal{X}^{\sigma,\beta}}\leq C\sqrt{m+\log(1+|t|)}.
		\end{equation}
		Moreover, there exists $c>0$ such that for all $t_0\in\R$, $m\in\N$, $N\geq 1$, 
		\begin{equation}\label{inclusion1}
		\Phi_N(t_0)(\wtilde{\Sigma}_N^m)\subset \wtilde{\Sigma}_N^{m+\left\lfloor\frac{\log(1+|t_0|)}{\log 2}\right\rfloor+2}.
		\end{equation}
	\end{proposition}
	\begin{proof}
		We follow closely \cite{BTT-AIF}. Define the set
		\begin{equation*}
			B_N^{m,k}(D):=\{(u_0,v_0)\in E_N\times E_N:\|(u_0,v_0)\|_{\mathcal{Y}^{\beta}}\leq D\sqrt{m+k}  \},
		\end{equation*} 
		where $D\gg 1$ is to be chosen later. From Proposition \ref{local_theory}, the time for local existence is
		$$ \tau_{m,k}:=ce^{-CD\sqrt{m+k}} \textrm{ on } B_N^{m,k}(D).
		$$ 
		Then for any $|t|\leq \tau_{m,k}$,
		\begin{equation}\label{A}
		\Phi_N(t)(B_N^{m,k}(D))\subset \{(u_0,v_0)\in E_N\times E_N: \sup_{|t|\leq \tau_{m,k}}\|\Phi_N(t)(u_0,v_0)\|_{\mathcal{X}^{\sigma,\beta}}\leq D\sqrt{m+k+1} \}.
		\end{equation}
		Moreover, 
		$$ \rho_N(E_N\times E_N\setminus B_N^{m,k}(D))\leq \mu(E_N\times E_N\setminus B_N^{m,k}(D))\leq Ce^{-cD^2(m+k)},
		$$
		thanks to Proposition \ref{LD} and Proposition \ref{probabilistic2} with $N=0$ there.
		
		Now we set
		$$ \wtilde{\Sigma}_{N}^{m,k}(D):=\bigcap_{j=-\left\lfloor \frac{2^k}{\tau_{m,k}}\right\rfloor}^{\left\lfloor \frac{2^k}{\tau_{m,k}}\right\rfloor}\Phi_N(-j\tau_{m,k})(B_N^{m,k}(D)).
		$$
		Thanks to \eqref{A}, we obtain that the for any $(u_0,v_0)\in \wtilde{\Sigma}_N^{m,k}(D)$, $|t|\leq 2^k$, 
		$$ \|\Phi_N(t)(u_0,v_0)\|_{\mathcal{X}^{\sigma,\beta}}\leq D\sqrt{m+k+1}.
		$$
		Since the measure $\rho_N$ is invariant by the flow $\Phi_N(t)$, we obtain that
		\begin{equation}\label{B}
		\begin{split}
		\rho_N(E_N\times E_N\setminus \wtilde{\Sigma}_{N}^{m,k}(D))\leq &
		\sum_{|j|\leq \left\lfloor 2^k\tau_{m,k}^{-1}\right\rfloor}\rho_N(E_N\times E_N\setminus \Phi_N(-j\tau_{m,k})(B_N^{m,k}(D))) 
		\\ \leq& \frac{2^{k+2}}{\tau_{m,k}}\rho_N(E_N\times E_N\setminus B_{N}^{m,k}(D))\\
		\leq & \frac{1}{c}2^{k+2}e^{CD\sqrt{m+k}-cD^2(m+k)}\leq 2^{-(m+k)},
		\end{split}
		\end{equation}
		provided that $D$ large enough, independent of $m,k$ and $N$. 
		
		Next, we set
		$$ \wtilde{\Sigma}_N^m:=\bigcap_{k=1}^{\infty}\wtilde{\Sigma}_N^{m,k}(D).
		$$
		Thanks to \eqref{B}, we have
		$$ \rho_N(E_N\times E_N\setminus \wtilde{\Sigma}_N^{m})\leq 2^{-m}.
		$$
		Moreover, we have for any $t\in \R$, $(u_0,v_0)\in\wtilde{\Sigma}_N^m$,
		$$ \|\Phi_N(t)(u_0,v_0)\|_{\mathcal{X}^{\sigma,\beta}}\leq C\sqrt{m+2+\log(1+|t|)}.
		$$
		Let us turn to the proof of \eqref{inclusion1}. The point is that the indices $m,k$ are symmetric in the definition of $B_N^{m,k}(D)$. Fix $(u_0,v_0)\in\wtilde{\Sigma}_N^m$, and let $k_0\in\mathbb{Z}$ so that 
		$ 2^{k_0-1}\leq |t|\leq 2^{k_0}.
		$ 
		 From \eqref{bd1} , we know that for any $k\geq k_0$, 
		$$ \|\Phi_N(t)\Phi_{N}(t_0)(u_0,v_0)\|_{\mathcal{X}^{\sigma,\beta}}\leq D\sqrt{m+k+k_0+1},
		$$
		which by definition means that
		$ \Phi_N(t)\Phi_N(t_0)(u_0,v_0)\in B_N^{m+k_0+1,k}(D).
		$
		This implies that $$\Phi_N(t_0)(u_0,v_0)\in \wtilde{\Sigma}_N^{m+k_0+1,k},\quad \forall k\in\mathbb{Z}.$$
		Thus
		$$ \Phi_N(t_0)(u_0,v_0)\in \wtilde{\Sigma}_N^{m+k_0+1}\subset \wtilde{\Sigma}_N^{m+2+\left\lfloor\frac{\log(1+|t_0|)}{\log 2}\right\rfloor}.
		$$
		This completes the proof of Proposition~\ref{almostglobal}.
	\end{proof}
	
In what follows, we always fix $\beta>1$.	
For integers $m\geq  1$ and $N\geq 1$, define the cylindrical sets
	$$ \Sigma_N^m:=\{(u,v)\in\mathcal{X}^{\sigma,\beta}:\pi_N(u,v)\in \wtilde{\Sigma}_N^m \}.
	$$

	For each $m\geq 1$, we set
	$$ \Sigma^m:=\left\{(f,g)\in\mathcal{X}^{\sigma,\beta}:\exists N_p\rightarrow\infty,\exists(f_p,g_{p})\in\Sigma_{N_p}^m, \lim_{p\rightarrow\infty}\|\pi_{N_p}(f_p,g_p)-(f,g)\|_{\mathcal{H}^{\sigma}}=0\right\}.
	$$
	\begin{lemme}\label{limsup}
		$\Sigma^m$ is a closed subset of $\mathcal{H}^{\sigma}$ and
		$$\limsup_{N\rightarrow\infty}\Sigma_N^m：=\bigcap_{N=1}^{\infty}\bigcup_{N'=N}^{\infty}\Sigma_{N'}^m\subset \Sigma^m.
		$$
		Moreover,
	\begin{equation}\label{mi}
	 \rho(\Sigma^m)\geq \rho(\mathcal{H}^{\sigma})-2^{-m}.
	\end{equation}	
	\end{lemme}
\begin{remarque}
	Due to the lack of periodicity of the linear evolution, the definition of $\Sigma^m$ and the proof given below is a little different, compared to  \cite{BTT-AIF}. Indeed, the strong convergence in $\mathcal{H}^{\sigma}$ and the weak convergence in $\mathcal{X}^{\sigma,\beta}$ will 
	fullfil our need.  
\end{remarque}
	\begin{proof}
		For the closeness, take a sequence $(f^k,g^k)\in\Sigma^m$, such that $(f^k,g^k)\rightarrow (f,g)$, strongly in $\mathcal{H}^{\sigma}$. By definition, for each $k$, there exists a sequence  $(f_{p_k}^k,g_{p_k}^k)\in \Sigma_{N_k}^{m}$, such that
		$$ \|\pi_{N_k}(f_{p_k}^k,g_{p_k}^k)-(f^k,g^k)\|_{\mathcal{X}^{\sigma,\beta}}<2^{-k}.
		$$
		In particular, $(f_{p_k}^k,g_{p_k}^k)\rightarrow (f,g)$, strongly in $\mathcal{H}^{\sigma}$. Hence $(f,g)\in\Sigma^m$.
		
		Next, for any $(f,g)\in \limsup_{N\rightarrow\infty} \Sigma_N^m$, there exists $N_p\rightarrow\infty$, such that $\pi_{N_p}(f,g)\in\widetilde{\Sigma}_{N_p}^m$. Denote by $(f_p,g_p):=\pi_{N_p}(f,g)$. Obviously we have $\|(f_p,g_p)-(f,g)\|_{\mathcal{H}^{\sigma}}\rightarrow 0$ as $p\rightarrow\infty$.
		The next goal is to show that $(f,g)\in\mathcal{Y}^{\beta}$. First we claim that from 
		$$ \sum_{l\in\Z}\langle l\rangle^{-\beta}\sup_{l\leq t<l+1}\|\pi_{N_p}S(t)(f,g)\|_{W^{\epsilon_0,r_0}(M)}\leq C\sqrt{m},
		$$
		and the fact that $\|\pi_{N}\|_{W^{\epsilon_0,r_0}(M)\rightarrow W^{\epsilon_0,r_0}(M)}\leq C$, we have $\|(f,g)\|_{\mathcal{Y}^{\beta}}\leq C\sqrt{m}$. 
		
		Indeed, for any fixed $l\in\Z$, $\pi_{N_p}S(t)(f,g)\rightarrow S(t)(f,g)$, strongly in $L^{\infty}([l,l+1]; \mathcal{H}^{\sigma})$. From Banach-Alaoglu theorem, $\pi_{N_p}S(t)(f,g)$ converges in weak* topology of the space $L^{\infty}([l,l+1];W^{\epsilon_0,r_0}(M))$, up to a subsequence. Thus
		   $\pi_{N_p}S(t)(f,g)*\rightharpoonup S(t)(f,g)$ in $L^{\infty}([l,l+1];W^{\epsilon_0,r_0}(M))$. Moreover, we obtain that
		$$ \|S(t)(f,g)\|_{L^{\infty}([l,l+1];W^{\epsilon_0,r_0}(M))}\leq \liminf_{p\rightarrow\infty} \sup_{l\leq t<l+1}\|\pi_{N_p}S(t)(f,g)\|_{W^{\epsilon_0,r_0}(M)}.
		$$
		Applying Fatou's lemma to the summation over $l\in\Z$, we conclude that
		$$ \sum_{l\in\Z}\langle l\rangle^{-\beta}\sup_{l\leq t<l+1}\|S(t)(f,g)\|_{W^{\epsilon_0,r_0}(M)}\leq C\sqrt{m}.
		$$
		
	To prove \eqref{mi}, we use Fatou's lemma to get
		$$ \rho(\Sigma^m)=\rho\left(\limsup_{N\rightarrow\infty}\Sigma_N^m\right)\geq \limsup_{N\rightarrow\infty}\rho(\Sigma_N^m).
		$$
		By definition,
		$$\rho(\Sigma_N^m)=\int_{\Sigma_N^m}G(u)d\mu(u),\quad \rho_N(\Sigma_N^m)=\int_{\Sigma_N^m}G_N(u)d\mu_N(u).
		$$
		From Lemma \ref{Gibbs_definition}, we know that 
		$$ \lim_{N\rightarrow\infty}(\rho(\Sigma_N^m)-\rho_N(\Sigma_N^m))=0.
		$$
		Thanks to Proposition \ref{almostglobal}, $\rho_N(\Sigma_N^m)\geq\rho_N(\mathcal{H}^{\sigma})-2^{-m}$, thus
		$$ \limsup_{N\rightarrow\infty}\rho(\Sigma_N^m)\geq \limsup_{N\rightarrow\infty}(\rho_N(\mathcal{H}^{\sigma})-2^{-m})=\rho(\mathcal{H}^{\sigma})-2^{-m}.
		$$
	This completes the proof of Lemma~\ref{limsup}.
	\end{proof}

	As a consequence, thet set
	\begin{equation}\label{definition-Sigma}
	 \Sigma:=\bigcup_{m=1}^{\infty}\Sigma^m,
	\end{equation}
	has the full $\rho$ measure.
	Following similar argument in \cite{BTT-AIF}, we have the following global existence results for the Cauchy problem \eqref{sys} with any initial condition $(u_0,v_0)\in \Sigma$.
	\begin{proposition}\label{globalexistence}
		For every integer $m\in\mathbb{N}$, the local solution $u$ of \eqref{sys} with initial condition $(u_0,v_0)\in\Sigma^m$ is globally defined and we will denote it by $(u,\partial_t u)=\Phi(t)(u_0,v_0)$. Moreover, there exists $C>0$ such that for every $(u_0,v_0)\in\Sigma^m$ and every $t\in\mathbb{R}$, we have 
		$$ \|\Phi(t)(u_0,v_0)\|_{\mathcal{X}^{\sigma,\beta}}\leq C\sqrt{m+\log(1+|t|)}.
		$$
		Furthermore, there exists $(u_{0,k},v_{0,k})\in\widetilde{\Sigma}_{N_k}^m, N_k\rightarrow\infty$, such that
		$$ \lim_{k\rightarrow\infty}\|\pi_{N_k}\Phi_{N_k}(t)(u_{0,k},v_{0,k})-\Phi(t)(u_0,v_0)\|_{\mathcal{H}^{\sigma}}=0,
		$$
		and for every $t\in\mathbb{R}$,
		\begin{equation}\label{limit_globalexistence}
		\lim_{k\rightarrow\infty}\|W(t)(u_0,v_0)-\pi_{N_k}W_{N_k}(t)(u_{0,k},v_{0,k})\|_{H^{\alpha}(M)}=0.
		\end{equation}
		Finally, for every $t\in\mathbb{R}, \Phi(t)(\Sigma)=\Sigma$.
	\end{proposition}
	
	\begin{remarque}
		Unlike \cite{BTT-AIF}, here we only have the strong convergence for the norm $\mathcal{H}^{\sigma}$ for linear evolution instead of the norm $\mathcal{X}^{\sigma,\beta}$, a counterpart of the norm $Y^s$ in \cite{BTT-AIF}. However, weak convergence holds true in the functional space $\mathcal{X}^{\sigma,\beta}$, which allows us to get the desired bound of the $\mathcal{X}^{\sigma,\beta}$ for the limit. 
	\end{remarque}

	\begin{proof}
		By assumption, there exist sequences $N_k\rightarrow\infty, (u_{0,k},v_{0,k})\in \widetilde{\Sigma}_{N_k}^m$ such that
		
		$$ \lim_{k\rightarrow\infty}\|\pi_{N_k}(u_{0,k},v_{0,k})-(u_0,v_0)\|_{\mathcal{H}^{\sigma}}=0.
		$$ 
		From Proposition \ref{almostglobal}, we know that for any $t\in\mathbb{R}$,
		\begin{equation}\label{-1}
		\|\pi_{N_k}\Phi_{N_k}(t)(u_{0,k},v_{0,k})\|_{\mathcal{X}^{\sigma,\beta}}\leq C\sqrt{m+\log (1+|t|)}.
		\end{equation}
		Denote by $\Lambda_T=C\sqrt{m+\log (1+|T|)}$ for any given $T>0$. In order to apply Lemma \ref{localconvergence}, we need show that there exists a uniform constant $C'>0$, such that
		\begin{equation}\label{goal}
		\|\Phi(t)(u_0,v_0)\|_{L^{\infty}([-T,T];\mathcal{X}^{\sigma,\beta})}\leq C'\Lambda_T.
		\end{equation} 
		Note that we could not obtain \eqref{goal} by passing to the limit of \eqref{-1}, since we do not know the whether  $\|(u_{0,k},v_{0,k})-(u_0,v_0)\|_{L^{\infty}([-T,T];\mathcal{Y}^{\beta})}$ converges to zero. 
		
		Since $(u_{0,k},v_{0,k})\rightarrow(u_0,v_0)$ strongly in $\mathcal{H}^{\sigma}$ and $\|\overline{S}(t)\|_{\mathcal{H}^{\sigma}\rightarrow\mathcal{H}^{\sigma}}\leq 1$, we have that $S(t)(u_{0,k},v_{0,k})\rightarrow S(t)(u_0,v_0)$ in $C(\R;H^{\sigma}(M))$. From $$\sum_{l\in\Z}\langle l\rangle^{-\beta}\sup_{l\leq s<l+1}\|S(t+s)(u_{0,k},v_{0,k})\|_{W^{\epsilon_0,r_0}(M)}\leq \Lambda_T, $$ Banach-Alaoglu theorem implies that $S(t)(u_{0,k},v_{0,k})^*\rightharpoonup S(t)(u_0,v_0)$, in the weak* topology of $L^{\infty}([l,l+1];W^{\epsilon_0,r_0}(M))$, up to a subsequence in priori. Note that the convergence indeed takes place for the full sequence, since it converges in the strong topology of $L^{\infty}([l,l+1];H^{\sigma}(M))$. Consequently,  we have
		$$ \|S(t)(u_0,v_0)\|_{L^{\infty}([l,l+1];W^{\epsilon_0,r_0}(M))}\leq \liminf_{k\rightarrow\infty}\|S(t)(u_{0,k},v_{0,k})\|_{L^{\infty}([l,l+1];W^{\epsilon_0,r_0}(M))}.
		$$
		Multiply by $\langle l\rangle^{-\beta}$ in both sides of the inequality above and sum over $l\in\Z$, we have that
		$$ \|(u_0,v_0)\|_{\mathcal{Y}^{\beta}}\leq \Lambda_T,
		$$
		thanks to Fatou's lemma .

		Now, we are in a position to apply Lemma \ref{localconvergence} from for $|t|\leq \tau=\tau_{m,T}:=ce^{-\kappa\Lambda_T}$. The outputs are 
		$$ \lim_{k\rightarrow\infty}\|W_{N_k}(t)(u_{0,k},v_{0,k})-W(t)(u_0,v_0)\|_{L^{\infty}([-\tau,\tau];H^{\alpha}(M))}=0,
		$$
		$$ \lim_{k\rightarrow\infty}\|\pi_{N_k}\Phi_{N_k}(\tau)(u_{0,k},v_{0,k})-\Phi(\tau)(u_0,v_0)\|_{L^{\infty}([-\tau,\tau];\mathcal{H}^{\sigma})}=0.
		$$
		The same weak convergence argument yields
		$$ \|\Phi(t)(u_0,v_0)\|_{\mathcal{Y}^{\beta}}\leq 2\Lambda_T,\quad \forall |t|\leq \tau.
		$$

		We can then apply Lemma \ref{localconvergence} successively with the same constant $\Lambda_T$, to continue the flow map $\Phi(t)(u_0,v_0)$ to $|t|\leq 2\tau, 3\tau,\cdots$ up to $T$. The key point is that at each iteration step, the argument above does not increase the constant $\Lambda_T$, hence the length of local existence can be always chosen as $\tau$.
		
		In order to check the invariance of the set $\Sigma$, note that from Proposition \ref{almostglobal}, 
		$$\Phi_{N_k}(t_0)(\widetilde{\Sigma}_{N_k}^m)\subset \widetilde{\Sigma}_{N_k}^{m+l(t_0)},\quad l(t_0)=\left\lfloor\frac{\log(1+|t_0|)}{\log 2}\right\rfloor+2.
		$$ 
		Thus for any $(u_0,v_0)\in \Sigma^m$, there exists a seuqence $(u_{0,N_k},v_{0,N_k})\in\widetilde{\Sigma}_{N_k}^m$, such that 
		$$ \lim_{k\rightarrow\infty}\|\pi_{N_k}\Phi(t_0)(u_{0,k},v_{0,k})-\Phi(t_0)(u_0,v_0)\|_{\mathcal{H}^{\sigma}}=0.
		$$ 
		By definition, this implies that $\Phi(t_0)(u_0,v_0)\in \Sigma^{m+l(t_0)}$. Thus $\Phi(\Sigma)\subset\Sigma$. From the reversibility of $\Phi(t)$, we have $\Phi(t)(\Sigma)=\Sigma$. This completes the proof of Proposition  \ref{globalexistence}.
	\end{proof}


	\subsection{Measure invariance}
	The proof follows from several reductions and an approximation lemma. One should pay attention to the topology used here.

	By reversibility, it suffices to show that
	\begin{equation}\label{invariance}
	\rho(A)\leq \rho(\Phi(t)A)
	\end{equation}
	for all $t\in \mathbb{R}$ and every measurable set $A\subset \Sigma\subset\mathcal{X}^{\sigma,\beta}$. Note that the flow $\Phi(t)$ is well-defined on $\Sigma\subset\mathcal{X}^{\sigma,\beta}$. By inner regularity of the measure $\mu$ (hence for $\rho$ ), there exists a sequence of closed set $F_n\subset A$, with respect to the topology of $\mathcal{H}^{\sigma}$, such that
	$$ \rho(A)=\lim_{n\rightarrow\infty}\rho(F_n).
	$$
	Note that $\Sigma $ and $\mathcal{X}^{\sigma,\beta}$ both have the full  $\rho$ measure, hence it is reduced to prove \eqref{invariance} for all $F\subset \Sigma$, closed in $\mathcal{H}^{\sigma}$. Indeed, $F_n\subset A$ implies that  $\Phi(t)F_n\subset \Phi(t)A$, thus
	$$ \rho(A)=\lim_{n\rightarrow\infty}\rho(F_n)\leq \lim_{n\rightarrow\infty}\rho(\Phi(t)F_n)\leq \rho(\Phi(t)A).
	$$ 
	Next we reduce the matter to prove \eqref{invariance} for all $B\subset\Sigma$, closed in $\mathcal{H}^{\sigma}$, while bounded in $\mathcal{X}^{\sigma,\beta}$. Given $F\subset\Sigma$, closed in $\mathcal{H}^{\sigma}$, we set
	$ F_R:=F\cap B_R^{\mathcal{X}^{\sigma,\beta}},
	$ where we use the notation $B_R^{Y}$ to denote the ball of radius $R$ with respect to the norm of the specified Banach space $Y$. From the large deviation bound $\rho(F_R^c)\leq Ce^{-cR^2}$, we have that $\displaystyle{\rho(F)=\lim_{R\rightarrow\infty}\rho(F_R). }$ Therefore, if \eqref{invariance} is true for all such $F_R$, we immediately have
	$$ \rho(F)=\lim_{n\rightarrow\infty}\rho(F_R)\leq\lim_{n\rightarrow\infty}\rho(\Phi(t)F_R)\leq \rho(\Phi(t)F).
	$$   
	For the third step, we reduce to prove \eqref{invariance} for all  $K\subset\Sigma$, compact with respect to the $\mathcal{H}^{\sigma}$ topology, while bounded in the norm of $\mathcal{X}^{\sigma,\beta}$ by $R$.
	
	Indeed, 
	given $B\subset\Sigma\cap B_R^{\mathcal{X}^{\sigma}}$, we define the set
	$$ K_n:=\{(u,v)\in B:\|(u,v)\|_{\mathcal{H}^{\sigma'}}\leq n \},\quad \sigma<\sigma'<\alpha-\frac{d}{2}.
	$$
	From Rellich theorem, we know that $K_n$ are compact sets in $\mathcal{H}^{\sigma}$. From the large deviation bound of of the type $\rho(K_n^c)\leq Ce^{-cn^2}$, we know that
	$$ \rho(F)=\lim_{n\rightarrow\infty}\rho(K_n).
	$$
	Thus the same argument as above yields
	$$ \rho(B)=\lim_{n\rightarrow\infty} \rho(K_n)\leq\lim_{n\rightarrow\infty}\rho(\Phi(t)K_n)\leq \rho(\Phi(t)B).
	$$

	Finally we assume that $K\subset\Sigma$ is a compact set with respect to the topology of $\mathcal{H}^{\sigma}$, which is bounded by $R$ in the norm of $\mathcal{X}^{\sigma,\beta}.$ To prove \eqref{invariance} for $K$, 
	we need an approximation lemma:
	\begin{lemme}\label{approximation}
		There exists $ C_0>0$ such that the following holds true. For every $R>1$, $\epsilon>0$, every set $K  \subset B_R^{\mathcal{X}^{\sigma,\beta}}$, compact with resepct to the topology of $\mathcal{H}^{\sigma}$, there exists $N_0\geq 1$ such that for all $N\geq N_0$  $(u_0,v_0)\in K$, and all $|t|\leq \tau_{C_0R}=ce^{-\kappa C_0R}$, we have
		$$ \|\Phi(t)(u_0,v_0)-\Phi_N(t)\pi_N(u_0,v_0)\|_{\mathcal{H}^{\sigma}}<\epsilon.
		$$
	\end{lemme}
	\begin{proof}
		The proof is just a refinement of the proof of Lemma \ref{localconvergence}. 
		Denote by $\pi_N(u_0,v_0)=(u_{0,N},v_{0,N}),$ and write
		$$ \Phi_N(t)(u_{0,N},v_{0,N})=\overline{S}(t)(u_{0,N},v_{0,N})+(w_N(t),\partial_tw_N(t)),$$$$ \Phi(t)(u_0,v_0)=\overline{S}(t)(u_0,v_0)+(w(t),\partial_tw(t)).
		$$
		Note that
		$$ \|\overline{S}(t)(1-\pi_N)(u_0,v_0)\|_{\mathcal{H}^{\sigma}}=\|(1-\pi_N)(u_0,v_0)\|_{\mathcal{H}^{\sigma}}\rightarrow 0,\textrm{ as } N\rightarrow \infty.
		$$
		From compactness of $K$, this convergence is uniform. It remains to prove the uniform convergence of the nonlinear part  $\|(w_N(t),\partial_tw_N(t))-(w(t),\partial_tw(t))\|_{\mathcal{H}^{\sigma}}$.

		First note that
		$$ \sup_{|t|\leq \frac{1}{2}}\|\overline{S}(t)(u_{0},v_{0})\|_{W^{\epsilon_0,r_0}(M)}\leq R,\quad \sup_{|t|\leq \frac{1}{2}}\|\overline{S}(t)(u_{0,N},v_{0,N})\|_{W^{\epsilon_0,r_0}(M)}\leq C_0R, 
		$$
		with $C_0=\sup_{N}\|\pi\|_{W^{\epsilon_0,r_0}\rightarrow\infty W^{\epsilon_0,r_0}}$.
		 From local existence theory, for all $(u_0,v_0)\in K\subset B_R^{\mathcal{X}^{\sigma}}$, for any $N$, we have
		
		$$ \|(w_N(t),\partial_tw_N(t))\|_{L^{\infty}([-\tau_{C_0R},\tau_{C_0R}];\mathcal{H}^{\alpha})}\leq C_0R+1, \quad \|(w(t),\partial_tw(t))\|_{L^{\infty}([-\tau_R,\tau_R];\mathcal{H}^{\alpha})}\leq R+1.
		$$
		To boung $\|w_N(t)-w(t)\|_{H^{\alpha}(M)}$, as in the proof of Lemma \ref{localconvergence}, we have to estimate three contributions
		\begin{equation*}
			\begin{split}
				&\mathrm{I}(t):=\int_0^t\left\|e^{\pi_{N}(w_N(t')+S(t')(u_{0,N},v_{0,N}))}-e^{\pi_{N}(w(t')+S(t')(u_0,v_{0}))}\right\|_{L^2(M)}dt',\\
				&\mathrm{II}(t):=\int_0^t\left\|e^{w(t')+S(t')(u_0,v_0)}-e^{\pi_{N}(w(t')+S(t')(u_0,v_0))}\right\|_{L^2(M)}dt',\\
				&\mathrm{III}(t):=\int_0^t\left\|\pi_N^{\perp}e^{w(t')+S(t')(u_0,v_0)}\right\|_{L^2(M)}dt'.
			\end{split}
		\end{equation*}
		The sum of the three contributions can be bounded by
		$$C(R,\tau_{C_0R})\left(\int_0^t\|w_N(t')-w(t)\|_{L^2(M)}dt'+o_{N\rightarrow\infty}(1)\right),
		$$
	It is not enough to conclude. We note the $o_{N\rightarrow\infty}(1)$ consists of the expressions of the form
	\begin{equation*}
		\begin{split} &\|\pi_{N}^{\perp}\exp(w(t)+S(t)(u_0,v_0))\|_{L^{\infty}([-\tau_{C_0R},\tau_{C_0R}];L^2(M))}\\ + &\|S(t)(u_0-u_{0,N},v_0-v_{0,N})\|_{L^{\infty}([-\tau_{C_0R},\tau_{C_0R}];L^2(M))},
			\end{split}
	\end{equation*}
	and	the second term above converges unifromly to $0$ since $(u_0,v_0)$ varies in a compact set $K$. The first term above can be bounded by
		$$ N^{-\sigma}\|e^{w(t)+S(t)(u_0,v_0)}\|_{H^{\sigma}(M)}\leq CN^{-\sigma}e^{\|w(t)+S(t)(u_0,v_0)\|_{L^{\infty}(M)}}\|w(t)+S(t)(u_0,v_0)\|_{H^{\sigma}(M)},
		$$
		which converges to $0$ uniformly. The proof of Lemma~\ref{approximation} is complete.
	\end{proof}

Now we can complete the proof of \eqref{invariance} for $K$.	From local well-posedness, there exists $A>0$, such that for all $\epsilon>0$, and $|t|\leq \tau_R$, we have
	$$ \Phi_N(t)(K+B_{\epsilon}^{\mathcal{X}^{\sigma,\beta}})\subset \Phi_N(t)K+B_{A\epsilon}^{\mathcal{H}^{\sigma}}.
	$$
	Thus
	\begin{equation*}
		\begin{split}
			\rho(\Phi(t)(K)+B_{2\epsilon}^{\mathcal{H}^{\sigma}})\geq &\limsup_{N\rightarrow\infty}\rho_N(\Phi(t)(K)+B_{2\epsilon}^{\mathcal{H}^{\sigma}}) \quad \textrm{ Fatou's lemma }\\
			\geq &\limsup_{N\rightarrow\infty}\rho_N(\Phi_N(t)((K+B_{A^{-1}\epsilon}^{\mathcal{X}^{\sigma,\beta}})\cap E_N)) \quad \textrm{ Lemma \ref{approximation}}\\
			\geq & \limsup_{N\rightarrow\infty}\rho_N(\Phi_N(t)(K\cap E_N))\\
			\geq & \limsup_{N\rightarrow\infty}\rho_N(K\cap E_N) \quad \textrm{ invariance for finite dimensional system}\\ \geq& \rho(K).
		\end{split}
	\end{equation*}
	By taking $\epsilon\rightarrow 0$, we have that $\rho(\Phi(t)K)\geq \rho(K)$, for all $|t|\leq \tau_R$.
	Finally, for any $t\in\mathbb{R}$, we can conclude by iteration.

	\section{unique limit for smooth approximation: Proof of Theorem \ref{thm3} }
Fix $\beta>1$ in this section. Take $s>0,\delta>0$ such that $0<s+\delta<\alpha-\frac{d}{2}-\epsilon_0$, as in the Proposition \ref{probabilistic2}. Define the set
$$ \mathcal{K}_N=\{(u_0,v_0): N^s\|\Pi_N^{\perp}(u_0,v_0)\|_{\mathcal{Y}^{\beta}}\leq 1 \},\quad \mathcal{K}:=\bigcup_{N=1}^{\infty}\bigcap_{N'=N}\mathcal{K}_{N'}.
$$
Thanks to Proposition \ref{probabilistic2}, we have
$$ \mu(\mathcal{K}_N^c)\leq Ce^{-cN^{2\delta}}.
$$
Hence the convergence of the series
$$ \sum_{N=1}^{\infty}\mu(\mathcal{K}_N^c)<\infty
$$
implies that $\mu(\mathcal{K})=1$. Recall that $\Sigma $ defined as \eqref{definition-Sigma} in section 4 has full $\rho$ measure. It also has full $\mu$ measure, thanks to the fact that $\exp\left(\int_{M}e^u\right)$ is $\mu$ almost surely finite. Set
$$ \mathcal{G}:=\mathcal{K}\cap \Sigma,\quad \Omega_0=j^{-1}(\mathcal{G}),
$$
where $j:\Omega\rightarrow \mathcal{H}^{\sigma}$ is the canonical mapping defining the Gaussian measure $\mu$ on $\mathcal{H}^{\sigma}$. Note that $\mathcal{P}[\Omega_0]=1$, our goal is to show that for all $\omega\in \Omega_0$, the sequence of smooth solutions to \eqref{sys} with initial datum
$$ \Pi_N\left(u_{0}^{\omega}(x), v_{0}^{\omega}(x)\right):=\left(\sum_{\lambda_n\leq N} \frac{g_n(\omega)}{\langle \lambda_n\rangle^{\alpha}}\varphi_n(x),\sum_{\lambda_n\leq N}h_n(\omega)\varphi_n(\omega)\right)
$$
converges to the global solution constructed through Proposition \ref{globalexistence} with initial data
$$ \left(u_{0}^{\omega}(x), v_{0}^{\omega}(x)\right):=\left(\sum_{n=0}^{\infty} \frac{g_n(\omega)}{\langle \lambda_n\rangle^{\alpha}}\varphi_n(x),\sum_{n=0}^{\infty}h_n(\omega)\varphi_n(\omega)\right).
$$	
We first show that the convergence holds on a small time interval.
\begin{lemme}\label{convergence-shortime}
There exist $c>0, \kappa>0$, such that for all $R>0$, if
	$$ \|(u_0,v_0)\|_{\mathcal{X}^{\sigma,\beta}}\leq R,
\textrm{  and }
 \lim_{N\rightarrow\infty}\|\Pi_N^{\perp}(u_0,v_0)\|_{\mathcal{Y}^{\beta}}=0,
$$	
then
$$ \lim_{N\rightarrow\infty}\sup_{t\in I}\|\Phi(t)\Pi_N(u_0,v_0)-\Phi(t)(u_0,v_0)\|_{\mathcal{X}^{\sigma,\beta}}=0,
$$
where $I=[-ce^{-\kappa R}, ce^{-\kappa R}]$.
\end{lemme}

\begin{proof}
The proof is very similar as in the proof of Lemma \ref{localconvergence}. Denote by
$$ \Phi(t)\Pi_N(u_0,v_0)=(u_N(t),\partial_tu_N(t) ),\quad \Phi(t)(u_0,v_0)=(u(t),\partial_tu(t) ),
$$  
and write
$$ u_N(t)=S(t)\Pi_N(u_0,v_0)+w_N(t), \quad u(t)=S(t)(u_0,v_0)+w(t).
$$
The existence and uniqueness of the flow $\Phi(t)$ on $I$ is guaranteed by Proposition~\ref{local_theory}, provided that we take the same constants $c>0, \kappa>0$ as in that proposition. Consequently, we have
\begin{equation}\label{bd}
\|(w_N,\partial_tw_N)\|_{L^{\infty}(I;\mathcal{H}^{\alpha})}+\|(w,\partial_tw)\|_{L^{\infty}(I;\mathcal{H}^{\alpha})}\leq 2C.
\end{equation}

We first show that the convergence for the linear evolution part. Obviously, 
$$ \lim_{N\rightarrow \infty}\|S(t)\Pi_N^{\perp}(u_0,v_0)\|_{L^{\infty}(I;\mathcal{H}^{\sigma})}=0. 
$$
 For the $\mathcal{Y}^{\beta}$ norm, by definition, 
$$ \sup_{t\in I}\|S(t)\Pi_N^{\perp}(u_0,v_0)\|_{\mathcal{Y}^{\beta}}\leq 2\|\Pi_{N}^{\perp}(u_0,v_0)\|_{\mathcal{Y}^{\beta}},
$$
hence it converges to zero. For the nonlinear part, using Duhamel, we write
$$ w_N(t)-w(t)=\int_{0}^t\frac{\sin((t-t')D^{\alpha})}{D^{\alpha}}\left(e^{S(t')\Pi_N(u_0,v_0)+w_N(t')}-e^{S(t')(u_0,v_0)+w(t')}\right)dt'.
$$
As in the proof of Lemma \ref{localconvergence}, we have 
\begin{equation*}
\begin{split}
&\|w_N(t)-w(t)\|_{H^{\alpha}(M)} \\ \leq  & \int_{0}^te^{\|S(t')\Pi_N(u_0,v_0)\|_{L^{\infty}(M)}+\|S(t')(u_0,v_0)\|_{L^{\infty}(M)}} \cdot e^{\|w_N(t')\|_{L^{\infty}(M)} }\|S(t')\Pi_N^{\perp}(u_0,v_0)\|_{L^2(M)}dt'\\
+&\int_{0}^te^{\|S(t')(u_0,v_0)\|_{L^{\infty}(M)}} \cdot e^{\|w_N(t')\|_{L^{\infty}(M)} +\|w(t')\|_{L^{\infty}(M)}     }\|w_N(t')-w(t)\|_{L^2(M)}dt'.
\end{split}
\end{equation*}
From Sobolev embedding $W^{\epsilon_0,r_0}\subset L^{\infty}$, $H^{\alpha}\subset L^{\infty}$, we have the bounds 
$$ \|S(t)\Pi_N(u_0,v_0)\|_{L^{\infty}(M)}\leq 2\|\Pi_N(u_0,v_0)\|_{\mathcal{Y}^{\beta}}\leq 2R, \quad \|w_N(t)\|_{L^{\infty}(M)}\leq \|w_N(t)\|_{H^{\alpha}(M)}\leq C,
$$
$$ \|S(t)(u_0,v_0)\|_{L^{\infty}(M)}\leq 2R, \quad \|w(t)\|_{L^{\infty}(M)}\leq C,
$$
thus
\begin{equation*}
\begin{split}
\|w_N(t)-w(t)\|_{H^{\alpha}(M)}\leq &C_1(R)\|S(t)\Pi_N^{\perp}(u_0,v_0)\|_{L^{\infty}(I;L^2(M))}\\+ &C_2(R)\int_0^t\|w_N(t)-w(t)\|_{H^{\alpha}(M)}dt.
\end{split}
\end{equation*}
This implies that
$$ \lim_{N\rightarrow\infty}\|w_N(t)-w(t)\|_{L^{\infty}(I;H^{\alpha}(M))}=0,
$$
thanks to Grwonwall inequality. Applying similar argument to $\partial_tw_N-\partial_tw$, we complete the proof of Lemma \ref{convergence-shortime}.
\end{proof}
The next lemma is the convergence on successive intervals. 
\begin{lemme}\label{convergence-iteration}
With the same $c,\kappa>0$ in Lemma \ref{convergence-shortime}, the following holds true. For any $R>0$, if 
$$ \sup_{t_0\leq |t|\leq t_0+1}\|\Phi(t)(u_0,v_0)\|_{\mathcal{X}^{\sigma,\beta}}\leq R,
$$
and
$$ \lim_{N\rightarrow\infty}\|\Phi(t_0)\Pi_N(u_0,v_0)-\Phi(t_0)(u_0,v_0)\|_{\mathcal{X}^{\sigma,\beta}}=0,
$$
then for $I=[t_0-ce^{-\kappa R}, t_0+ce^{-\kappa R}]$, we have
$$ \lim_{N\rightarrow\infty}\sup_{t\in I}\|\Phi(t)\Pi_N(u_0,v_0)-\Phi(t)(u_0,v_0)\|_{\mathcal{X}^{\sigma,\beta}}=0.
$$
\end{lemme}
\begin{proof}
The proof is similar. Denote by
$$ \Phi(t)\Pi_N(u_0,v_0)=(u_N(t),\partial_tu_N(t)), \quad \Phi(t)(u_0,v_0)=(u(t),\partial_tu(t)).
$$
We write
$$ u_N(t)=S(t-t_0)\Phi(t_0)\Pi_N(u_0,v_0)+w_N(t), \quad u(t)=S(t-t_0)\Phi(t_0)(u_0,v_0)+w(t).
$$
For the linear evolution part, we observe that
\begin{equation*}
\begin{split}
&\sup_{t\in I}\|S(t-t_0)(\Phi(t_0)\Pi_N(u_0,v_0)-\Phi(t_0)(u_0,v_0)) \|_{\mathcal{Y}^{\beta}}\\
=&\sup_{t\in I}\sum_{l\in\Z}(1+|l|)^{-\beta}\sup_{s\in [l,l+1]}\|S(s+t-t_0)(\Phi(t_0)\Pi_N(u_0,v_0)-\Phi(t_0)(u_0,v_0))\|_{W^{\epsilon_0,r_0}(M)}\\
\leq & 2\|\Phi(t_0)\Pi_N(u_0,v_0)-\Phi(t_0)(u_0,v_0)\|_{\mathcal{Y}^{\beta}},
\end{split}
\end{equation*}
which converges to $0$ by assumption. Next, from local Cauchy theory, the solution $\Phi(t)\Pi_N(u_0,v_0)$ concides with the solution construced by Proposition \ref{local_theory} with initial data $\Phi(t_0)\Pi_N(u_0,v_0)$ for $t\in I$, thus
$$ \sup_{t\in I}\|\Phi(t)\Pi_N(u_0,v_0)\|_{\mathcal{X}^{\sigma,\beta}}\leq 2R
$$
for $N$ large enough. Now the convergence for the nonlinear part can be obtained in the same way as in the proof of Lemma \ref{convergence-shortime}. This ends the proof of Lemma \ref{convergence-iteration}.
\end{proof}
Now we finish the proof of Theorem \ref{thm3}.
\begin{proof}[Proof of Theorem \ref{thm3}]
Fix $T>0$, and $(u_0,v_0)\in \Sigma^k\cap\mathcal{K}$ for some $k\in\N$, it will be sufficient to show that
\begin{equation}\label{convergence-longtime}
 \lim_{N\rightarrow\infty}\sup_{|t|\leq T}\|\Phi(t)\Pi_N(u_0,v_0)-\Phi(t)(u_0,v_0)\|_{\mathcal{X}^{\sigma,\beta}}=0.
\end{equation}
By definition of $\Sigma^k$ and Proposition \ref{globalexistence}, for all $|t|\leq T$, 
\begin{equation}\label{bdPhi(t)}
\|\Phi(t)(u_0,v_0)\|_{\mathcal{X}^{\sigma,\beta}}\leq C\sqrt{k+\log(1+T)}.
\end{equation} 
Set
$$ R_{k,t}:=C\sqrt{k+\log(1+|t|)},\quad
   \tau=\tau_{k,T}:=ce^{-\kappa R_{k,2T}}
$$
as in the Lemma \ref{convergence-shortime} and Lemma \ref{convergence-shortime}.
Applying Lemma \ref{convergence-shortime} from $[0,\tau]$, we obtain that
$$ \lim_{N\rightarrow\infty}\|\Phi(t)\Pi_N(u_0,v_0)-\Phi(t)(u_0,v_0)\|_{L^{\infty}([0,\tau];\mathcal{X}^{\sigma,\beta})}=0.
$$
We can then successively apply Lemma~\ref{convergence-iteration} with $R=R_{k,2T}$ on $$[\tau,2\tau],[2\tau,3\tau],\cdots, [(m_0-1)\tau,m_0\tau]$$for $T\leq m_0\tau \leq T+\tau$. The key point here is that we have a uniform bound for $\|\Phi(t)(u_0,v_0)\|_{\mathcal{X}^{\sigma,\beta}}$ on each interval. Once we prove the convergence on $[m\tau,(m+1)\tau]$, the initial convergence condition holds for the successive interval. Finally, we conclude 
$$ \lim_{N\rightarrow\infty}\sup_{0\leq t\leq  T}\|\Phi(t)\Pi_N(u_0,v_0)-\Phi(t)(u_0,v_0)\|_{\mathcal{X}^{\sigma,\beta}}=0.
$$
Using the same argument for the negative time $t\in[0,-T]$, we conclude the proof of Theorem~ref{thm3}.
 
\end{proof}


	\section{Ill posedness for power type nonlinearity}
	Here we consider the ill posedness of the equation 
	$$ \partial_t^2u+D^{2\alpha}u+u^{2k+1}=0,
	$$
	for $x\in\mathbb{T}^d$ and $\alpha\geq\frac{d}{2}.$
	We are going to prove the following ill-posedness result.
	
	\begin{proposition}\label{ill-proposition}
		\begin{enumerate}
			\item If $0<s<\frac{d}{2}-\frac{\alpha}{k}$,  then there exist a seqeunce $u_n(t,x)\in C(\mathbb{R};C^{\infty}(M))$, solutions to
			$$ \partial_t^2u_n+D^{2\alpha}u_n+u_n^{2k+1}=0, \quad (u_n,\partial_tu_n)|_{t=0}=(u_{n,0},u_{n,1}),
			$$
			and a sequence $t_n$ tending to zero, such that
			$$ \|(u_{n,0},u_{n,1})\|_{{\mathcal H}^s}\rightarrow 0,
			$$
			while
			$$\|u_n(t_n)\|_{H^s(M)}\rightarrow\infty.
			$$
			\item In addition, if $\frac{d}{2}-\frac{2\alpha}{2k-1}\leq s<\frac{d}{2}-\frac{\alpha}{k}$, then for any $(u_0,u_1)\in C^{\infty}(M)\times C^{\infty}(M)$, there exists a seqeunce $u_n(t,x)\in C(\mathbb{R};C^{\infty}(M))$, solutions to
			$$ \partial_t^2u_n+D^{2\alpha}u_n+u_n^{2k+1}=0, \quad (u_n,\partial_tu_n)|_{t=0}=(u_{n,0},u_{n,1}),
			$$
			and a sequence $t_n$ tending to zero,
			such that
			$$ \|(u_{n,0},u_{n,1})-(u_0,u_1)\|_{{\mathcal H}^s}\rightarrow 0,
			$$
			while
			$$ \|u_n(t_n)\|_{H^s(M)}\rightarrow\infty.
			$$
		\end{enumerate}
	\end{proposition}
	Note that to pass from $(u_0,u_1)\in \mathcal{H}^s(M)$, one can use the diagonal argument as in \cite{Tz-Lecture}. 	We refer to \cite{BT1/2,CCT,L,Oh_new,xia} for similar ill-posedness results of dispersive equations. The novelty for the present proposition is that under the additional assumption:
	$\frac{d}{2}-\frac{2\alpha}{2k-1}\leq s$, we obtain norm inflation near any smooth data $(u_0,u_1)\in C^{\infty}(M)\times C^{\infty}(M)$. 
	
	Before proving Proposition \ref{ill-proposition}, we remark that our construction is purely local, hence we may assume that $M=\T^d$ to avoid needless complications in the argument. The proof of Proposition \ref{ill-proposition} will be divided into several lemmas which we will establish in the rest of this section.

	\subsection{Unstable ODE profile}
Before constructing unstable profile, we need a lemma.
\begin{lemme}\label{generalODE}
Assume that $F\in C^2(\R;\R_+)$, $F'(v)=f(v)$ and $f(0)=0, f^{'}(v)\geq 0$. Then the solution $V$ of the ordinary differential equation
	\begin{equation}\label{ODEgeneral}
	\begin{cases}
	& V''+f(V)=0,\\
	&V(0)=V_0>0, F(V_0)>F(0),  V'(0)=0
	\end{cases}
	\end{equation}
	is a globally defined smooth function. Moreover,  $t\mapsto V(t)$ is periodic.
\end{lemme}

\begin{proof}
	The proof is standard, and the key points are the following facts:
	\begin{itemize}
		\item $v\mapsto F(v)$ is decreasing for $v<0$ and increasing for $v>0$. For any $a>F(0)$, there are exactly two roots $v_-<0,v_+>0$ such that $F(v_{\pm})=a$.
		\item First integral:
		$$ |V'(t)|^2=2(F(V_0)-F(V(t)).
		$$
	\end{itemize}
From the
	$ V''=-f(V)
	$,
	we know that $V''(0)<0$, hence $V'(t)<0$ for $t>0$ small. Thus
	there exists $t_0>0$ such that $V(t_0)=0$, and $V'(t_0)<0, V''(t_0)=0$. For $t$ slightly larger than $t_0$, $t\mapsto V'(t)$ is increasing while remaining negative. Therefore, there exists $T_1\in(t_0,+\infty]$, such that 
	$$  V'(T_1)=0,\quad V'(t)<0 ,\quad \forall t<T_1.
	$$
	Moreover, we have
	$$ V(T_1)<0,\quad F(V(T_1))=F(V_0)>F(0).
	$$
	From the first integral of $V$, we have
	$$ T_1=\int_{V(T_1)}^{V_0}\frac{dv}{\sqrt{2(F(V_0)-F(v))}}.
	$$ 
	Using the fact that $|f(V_0)|\neq 0$, we deduce that the integral above is finite, thus $T_1<\infty$. Then from the same argument, one conclude that there exists $T\in (T_1,+\infty)$, such that $V'(T)=0, V(T_1)=V_0$. Moreover,
	$$ T-T_1=\int_{V(T_1)}^{V_0}\frac{dV}{\sqrt{2(F(V_0)-F(v))}}=T_1.
	$$
	This implies that the function $V(t)$ is periodic with period $T=2T_1$. This completes the proof of Lemma \ref{generalODE}.
\end{proof}
	
	We will use Lemma \ref{ODEgeneral} to the special case $f(v)=v^{2k+1}, F(v)=\frac{1}{2k+2}v^{2k+2}$. Denote by $V(t)$ the global periodic solution of the ODE
	\begin{equation}\label{ODE-polynome}
V''+V^{2k+1}=0, \quad V(0)=1,V'(0)=0.
\end{equation}
	 We construct nonlinear profiles
	$$ \partial_t^2v_n+v_n^{2k+1}=0, \quad (v_n,\partial_tv_n)|_{t=0}=(\kappa_nn^{\frac{d}{2}-s}\varphi(nx),0),
	$$
	with
	$$ \kappa_n=(\log n)^{-\delta_1}
	$$ 
	to be chosen later. From a scaling property, we have
	\begin{equation}\label{explicit}
	v_n(t,x)=\left(\kappa_nn^{\frac{d}{2}-s}\varphi(nx)\right)V\left(t\left(\kappa_nn^{\frac{d}{2}-s}\varphi(nx)\right)^{k}\right),
	\end{equation}
	Moreover, we have
	$$ \nabla_xv_n(t,x)=\left(\kappa_nn^{\frac{d}{2}-s}n(\nabla\varphi)(nx)\right)\left[kt\lambda_n\varphi(nx)^{k}V'(\lambda_n t\varphi(nx)^{k})+V(\lambda_n t\varphi(nx)^{k})\right].
	$$
	where
	$$ \lambda_n=\left(\kappa_nn^{\frac{d}{2}-s}\right)^{k}.
	$$
	We now estimate various Sobolev norms of $v_n$.
	\begin{lemme}\label{Sobolevbound}
		Let
		$$ t_n=\left((\log n)^{\delta_2}n^{-(\frac{d}{2}-s)}\right)^{k}
		$$
		for some $\delta_2>\delta_1$, then we have for all $t\in[0,t_n]$,
		\begin{enumerate}
			\item $$ \|v_n(t_n)\|_{H^s}\geq c \kappa_n (\lambda_n t_n)^{s}.
			$$
			\item
			$$ \|v_n(t)\|_{H^{\sigma}}\leq C \kappa_n (\lambda_nt_n)^{\sigma}n^{\sigma-s},\forall \sigma=0,1,2,\cdots. 
			$$
			\item 
			$$ \|v_n(t)\|_{L^{\infty}}\leq C \lambda_n^{\frac{1}{k}}= C (\log n)^{-\delta_1}n^{\frac{d}{2}-s}.
			$$
		\end{enumerate}
	\end{lemme}
	\begin{proof}
We only prove (1), since the upper bounds follow directly.	
	
		\textbf{Case A: $1<s<s_c=\frac{d}{2}-\frac{\alpha}{k}$:}
		From
		$$ \|v_n(t)\|_{H^1}\leq \|v_n(t)\|_{L^2}^{1-\frac{1}{s}}\|v_n(t)\|_{H^s}^{\frac{1}{s}},
		$$
		we need majorize $\|v_n(t)\|_{L^2}$ and minorize $\|v_n(t)\|_{H^1}$. For the upper bound, from the construction of $v_n$, we have
		$$ \|v_n(t)\|_{L^2}\lesssim \kappa_nn^{-s},\quad 
		 \|\nabla v_n(t)\|_{L^2}\lesssim \lambda_nt \kappa_nn^{1-s}.
		$$
		For the lower bound, it would be sufficient to minorize the dominant part in $\nabla v_n $ (up to some constant) 
		$$ \lambda_n t\kappa_n n^{1-s} \cdot n^{\frac{d}{2}}\left(\varphi^{k}\nabla\varphi V'(\lambda_nt \varphi^k)\right)(nx).
		$$
	To bound it from below, unlike in \cite{Tz-Lecture}, we present a geometric argument.	
		\begin{lemme}\label{geometric}
		Assume that $\psi\in C_c^{\infty}(\R^d)$ and $\psi(x)>0$ for all $|x|<1$. Let $W$ be a periodic continuous function, $W\neq 0$. Then there exist $c_0>0, \lambda_0>0$, such that for all $\lambda\geq \lambda_0$, we have
		$$ \|\psi(x)W(\lambda\psi(x))\|_{L^2(\R^d)}\geq c_0>0.
		$$	
		\end{lemme}

	\begin{proof} 
		The proof of this lemma can be found in \cite{Tz-Lecture}. Here we present a different proof. 
		From the support property of $\psi$, there exist $0<a<b<1$, such that 
		$$
		\mathcal{C}_{a,b}:=\{x\in\mathbb{R}^d:a\leq |x|\leq b\}\subset \{x\in\mathbb{R}^d:|\nabla \psi(x)|\geq c_1 \}
		$$
		for some uniform constant $c_1>0$.  Denote by
		$ \Sigma_{s}:=\{x:\psi(x)=s\}.
		$
		After shrinking $a,b$ if necessary, we may assume that $\psi(\mathcal{C}_{a,b})$ is foliated by $\Sigma_s, s\in[a,b]$. Assume that
		$$ \max_{x\in\mathcal{C}_{a,b}}\psi(x)=B, \quad \min_{x\in\mathcal{C}_{a,b}}\psi(x)=A,
		$$
		By co-aera formula, we have for any continuous function $F$,
		$$ \int_{\mathcal{C}_{a,b}} F(\psi(x))dx=\int_{s\in[A,B]}F(s)ds
		\int_{\Sigma_s}\frac{d\sigma_{\Sigma_s}}{|\nabla\psi(x)|}.
		$$
		Therefore,
		\begin{equation}
		\begin{split}
		\|\psi(x)W(\lambda\psi(x))\|_{L^2(\R^d)}\geq & \int_{\mathcal{C}_{a,b}}|\psi(x)|^2|W(\lambda \psi(x))|^2dx\\
		\geq &\int_{s\in[A,B]} s^2|W(\lambda s)|^2ds\int_{\Sigma_s}\frac{d\sigma_{\Sigma_s}(x)}{|\nabla\psi(x)|}\\
		\geq c &\int_{s\in[A,B]}s^2|W(\lambda s)|^2dx,
		\end{split}
		\end{equation}
		where to the last inequality, we have used the continuity of the map
		$$ s\in[a,b]\mapsto \mathcal{M}^{d-1}(\Sigma_s),
		$$
		and $\mathcal{M}^{d-1}$ denotes the $d-1$ dimensional Hausdorff measure. 	Finally, by changing of variable, we obtain that
		$$ \int_{s\in[A,B]}s^2W(\lambda s)^2ds\geq \frac{c_0'}{\lambda^3}\int_{w\in[\lambda A,\lambda B]}w^2W(w)dw\geq c_0>0.
		$$
		This completes the proof of Lemma \ref{geometric}.
	\end{proof}
Lemma \ref{geometric} is sufficient to finish the analysis in Case A.

		\textbf{Case B:  $0<s<1$: }
	By definition, we have the trivial bound
		$$ \|v_n(t)\|_{H^2}\leq C \kappa_n (\lambda_n t)^2 n^{2-s}.
		$$
	From interpolation, we get
		$$ \|v_n(t)\|_{H^1}\leq C \|v_n(t)\|_{H^s}^{\frac{1}{2-s}}\|v_n(t)\|_{H^2}^{\frac{1-s}{2-s}}.
		$$
		Thus
		$$ \|v_n(t_n)\|_{H^s}\geq c \kappa_n(\lambda_n t_n)^s.
		$$
		This completes the proof of Lemma \ref{Sobolevbound}.
	\end{proof}

	\subsection{Energy estimate}

	Denote by $(u_n,\partial_tu_n)$ the solution to
	$$ \partial_t^2u+D^{2\alpha}u+u^{2k+1}=0,
	$$
	subject to the initial data
	$$ (u,\partial_tu)|_{t=0}=(u_0,u_1)+(v_{n}(0),0).
	$$
	\begin{lemme}\label{Energyestimate}
		Fix $k\in\mathbb{N}$ and $\alpha>\frac{d}{2}$.	There exists $\delta_2>\delta_1>0$ and $\epsilon>0$ such that for $t_n=(\log n)^{k\delta_2}n^{-k\left(\frac{d}{2}-s\right)}$, the following holds true.
		\begin{enumerate}
			\item If $\frac{d}{2}-\frac{2\alpha}{2k-1}\leq s<\frac{d}{2}-\frac{\alpha}{k}$, then for any $(u_0,u_1)\in C^{\infty}(\T^d)\times C^{\infty}(\T^d)$,
			$$ \sup_{t\in[0,t_n]}\|u_n(t)-S(t)(u_0,u_1)-v_n(t)\|_{H^s(M)}\leq n^{-\epsilon}.
			$$ 
			\item If $s<\frac{d}{2}-\frac{2\alpha}{2k-1}$, and $(u_0,u_1)=(0,0)$, we have
			$$ \sup_{t\in[0,t_n]}\|u_n(t)-v_n(t)\|_{H^s(M)}\leq n^{-\epsilon}.
			$$
		\end{enumerate}	
	\end{lemme}
	
	\begin{proof}
		Denote by 
		$$ w_n=u_n-S(t)(u_0,u_1)-v_n, \quad  (w_n,\partial_tw_n)|_{t=0}=(0,0).
		$$
		
		Define the semi-classical energy
		$$ E_n(w)(t):=\frac{1}{n^{2(\alpha-s)}}\|\partial_tw(t)\|_{L^2}^2+\frac{1}{n^{2(\alpha-s)}}\|D^{\alpha}w(t)\|_{L^2}^2.
		$$
We have that	$w_n$ satisfies the equation:
		$$ \partial_t^2w_n+D^{2\alpha}w_n=-D^{2\alpha}v_n-(f(v_n+u_L+w_n)-f(v_n)),
		$$
where
		$$ u_L(t)=S(t)(u_0,u_1).
		$$
		Multiplying by $\partial_tw_n$ and integrating over $\mathbb{T}^d$ to both side, we  obtain that
		\begin{equation}\label{energyderive1}
		\begin{split}
		\frac{1}{2}\frac{d}{dt}\left(\|\partial_tw_n(t)\|_{L^2}^2+\|D^{\alpha}w_n(t)\|_{L^2}^2\right)\leq & \left|\int_{\mathbb{T}^d}\partial_tw_n\cdot F_n(t)dx\right|\\
		\leq &\|\partial_tw_n\|_{L^2}\|F_n(t)\|_{L^2},
		\end{split}
		\end{equation}
	where
		$ F_n(t)=D^{2\alpha}v_n+(f(v_n+u_L+w_n)-f(v_n)).
		$
		Thus
		$$ \left|\frac{d}{dt}E_n(w_n(t))\right|\leq \frac{C}{n^{\alpha-s}}(E_n(w_n(t)))^{1/2}\|F_n(t)\|_{L^2}.
		$$
	To simplify the notation,	we denote by
		$$ e_n(t):=\sup_{0\leq \tau\leq t}E_n(w_n(\tau))^{1/2}.
		$$
		From Lemma \ref{Sobolevbound}, for all $0\leq t\leq t_n$,  
		$$ \|F_n(t)\|_{L^2}\leq C
		\kappa_n(\lambda_nt_n)^{2\alpha}n^{2\alpha-s}+\|G_n(t)\|_{L^2}\leq C(\log n)^{\frac{k}{\alpha}\delta_2-\left(\frac{k}{\alpha}+1\right)\delta_1}n^{2\alpha-s}+\|G_n(t)\|_{L^2},
		$$
		with
		$ G_n=f(v_n+u_L+w_n)-f(v_n).
		$
		Thanks to $u_L\in C^{\infty}(\R\times\T^d)$, we have
		
		\begin{equation}\label{Gn}
		|G_n|\leq C |w_n|(|v_n|^{2k}+|w_n|^{2k}+1)+C|u_L||v_n|^{2k}.
		\end{equation}
		By writing (using $w_n(0,x)=0$)
		$$ w_n(t,x)=\int_0^t\partial_tw_n(\tau,x)d\tau,
		$$
		we obtain that
		\begin{equation}
		\begin{split}
		\|G_n(t)\|_{L^2}\leq &C\int_0^t\|\partial_t w_n(\tau)\|_{L^2}d\tau\left(\|v_n(t)\|_{L^{\infty}}^{2k}+\|w_n(t)\|_{L^{\infty}}^{2k}\right)+n^{(2k-1)\left(\frac{d}{2}-s\right)-s}\\
		\leq& Cn^{\alpha-s}te_n(t)\left(\|v_n(t)\|_{L^{\infty}}^{2k}+\|w_n(t)\|_{L^{\infty}}^{2k}\right)+n^{(2k-1)\left(\frac{d}{2}-s\right)-s}.
		\end{split}
		\end{equation}
		From Lemma \ref{Sobolevbound}, we majorize the terms involving $v_n$ by
		$$ Cn^{\alpha-s}te_n(t)(\log n)^{-2k\delta_1}n^{2k\left(\frac{d}{2}-s\right)}.
		$$
		For the terms involving $w_n$, using the assumption that $\alpha>\frac{d}{2}$, we have 
		$$ \|w_n(t)\|_{L^{\infty}}\leq \|w_n(t)\|_{H^{\alpha}}^{\frac{d}{2\alpha}}\|w_n(t)\|_{L^2}^{1-\frac{d}{2\alpha}}\leq C(n^{\alpha-s}e_n(t))^{\frac{d}{2\alpha}}\left(n^{\alpha-s}te_n(t)\right)^{1-\frac{d}{2\alpha}}.
		$$ 
		Therefore,
		\begin{equation*}
			\begin{split}
				n^{-(\alpha-s)}\|G_n(t)\|_{L^2}\leq & Cn^{(2k-1)\left(\frac{d}{2}-s\right)-\alpha}\\+& Cte_n(t)n^{2k\left(\frac{d}{2}-s\right)}
				+ Cte_n(t)(n^{\alpha-s}e_n(t))^{2k}t^{2k\left(1-\frac{d}{2\alpha}\right)}
			\end{split}
		\end{equation*}
		
		Thus
		\begin{equation}\label{energybound1}
		\begin{split}
		\left|\frac{d}{dt}(E_n(w_n(t))^{1/2})\right|\leq & C(\log n)^{\frac{k}{\alpha}\delta_2-\left(\frac{k}{\alpha}+1\right)\delta_1}n^{\alpha}+Cn^{(2k-1)\left(\frac{d}{2}-s\right)-\alpha}\\+&Ct_ne_n(t)n^{2k\left(\frac{d}{2}-s\right)}
		+Ct_n^{1+2k\left(1-\frac{d}{2\alpha}\right)}e_n(t)^{2k+1}n^{2k(\alpha-s)}.
		\end{split}
		\end{equation}
		
		By monotonicity and definition, we observe that
		$$ \frac{de_n(t)}{dt}\leq \left|\frac{d}{dt}E_n(w_n(t))^{1/2}\right|.
		$$
		Thus
		
		\begin{equation}\label{energybound2}
		\begin{split}
		\frac{de_n}{dt}\leq & C(\log n)^{\frac{k}{\alpha}\delta_2-\left(\frac{k}{\alpha}+1\right)\delta_1}n^{\alpha}+Cn^{(2k-1)\left(\frac{d}{2}-s\right)-\alpha}\\+&Ct_ne_n(t)n^{2k\left(\frac{d}{2}-s\right)}
		+Ct_n^{1+2k\left(1-\frac{d}{2\alpha}\right)}e_n(t)^{2k+1}n^{2k(\alpha-s)}.
		\end{split}
		\end{equation}
Note that
		$$ t_n n^{2k\left(\frac{d}{2}-s\right)}\leq C n^{k\left(\frac{d}{2}-s\right)}, \quad t_n^{1+2k\left(1-\frac{d}{2\alpha}\right)}n^{2k(\alpha-s)}\leq C n^{k\left[\left(\alpha-\frac{d}{2}\right)+(\alpha-s)-\frac{2k}{\alpha}\left(\frac{d}{2}-s\right)\left(\alpha-\frac{d}{2}\right)\right]}
		$$
and
		\begin{equation*}
			\begin{split}
				&\left(\alpha-\frac{d}{2}\right)+(\alpha-s)-\frac{2k}{\alpha}\left(\frac{d}{2}-s\right)\left(\alpha-\frac{d}{2}\right)\\
				=&2\left(1-\frac{k}{\alpha}\left(\frac{d}{2}-s\right)\right)\left(\alpha-\frac{d}{2}\right)+\left(\frac{d}{2}-s\right)
				<\left(\frac{d}{2}-s\right),
			\end{split}
		\end{equation*}
		thanks to 
		$$  s<\frac{d}{2}-\frac{\alpha}{k},\quad \alpha>\frac{d}{2}.
		$$
		Consequently, we obtain that
		\begin{equation}\label{energybound3}
		\begin{split}
		\frac{de_n}{dt}\leq & C(\log n)^{\frac{k}{\alpha}\delta_2-\left(\frac{k}{\alpha}+1\right)\delta_1}n^{\alpha}+Cn^{(2k-1)\left(\frac{d}{2}-s\right)-\alpha}\\+&Cn^{k\left(\frac{d}{2}-s\right)}(e_n(t)+e_n(t)^{2k+1}).
		\end{split}
		\end{equation}	
	\textbf{Case 1: $\frac{d}{2}-\frac{2\alpha}{2k-1}\leq s<\frac{d}{2}-\frac{\alpha}{k}$.}
		
		In this case, 
		$$ n^{(2k-1)\left(\frac{d}{2}-s\right)-\alpha}<n^{\alpha},
		$$
		hence for all $t\in[0,t_n]$
		
		\begin{equation}\label{energybound2.1}
		\begin{split}
		\frac{de_n}{dt}\leq  C(\log n)^{\frac{k}{\alpha}\delta_2-\left(\frac{k}{\alpha}+1\right)\delta_1}n^{\alpha}+Cn^{k\left(\frac{d}{2}-s\right)}(e_n(t)+e_n(t)^{2k+1}).
		\end{split}
		\end{equation}
		Since $e_n(0)=0$, we may assume that $e_n(t)\leq 1$. Using Gronwall inequality and bootstrap argument, we finally obtain that
		$$ e_n(t)\leq C(\log n)^{\left(k\delta_2-\left(\frac{k}{\alpha}+1\right)\delta_1\right)}n^{\alpha-k\left(\frac{d}{2}-s\right)}e^{Ct(\log n)^{2k\delta_2}}\leq C n^{-\epsilon},
		$$
		provided that $\delta_2$ is small enough, thanks to $s<\frac{d}{2}-\frac{\alpha}{k}$. The remaining part follows from interpolation, since $E_n(w_n(t))$ essentially controls the norm $H^s$(for high frequencies).

	\textbf{Case 2: $s<\frac{d}{2}-\frac{2\alpha}{2k-1}.$}
		
		Notice that under the additional assumption that $(u_0,u_1)=0$, the term $|u_L||v_n|^{2k}$ in \eqref{Gn}
		does not appear. Hence we have the same estimate \eqref{energybound2.1}. The rest of the arguments are the same. The proof of Lemma \ref{Energyestimate} is now completed.
		\end{proof}	
As an immediate consequence of Lemma \ref{Energyestimate}, the proof of Proposition \ref{ill-proposition} is also completed.
	
	\section{ Polynomial nonlinearity in the case of a general randomisation}
	We study the fractional Klein-Gordon equation with polynomial nonlinearity:
	\begin{equation}\label{FNLKG}
	\partial_t^2u+D^{2\alpha}u+u^{2k+1}=0, \quad (u,\partial_tu)|_{t=0}=(u_0^{\omega},v_0^{\omega}),
	\end{equation}
	for general randomized initial data \eqref{randomization-general}. We will sketch the proof of Theorem \ref{thm-general}. Comparing to the previous situation where we globalize the solution via invariance of the Gibbs measure, now we construct global dynamics by energy method, in the spirit of \cite{BT-JEMS}. The key point is to establish a probabilistic energy a priori estimate.  
	
	To state the key proposition, we decompose the solution into linear evolution part and nonlinear part:
	$$ u(t)=z(t)+w(t),\quad  z(t)=S(t)(u_0,v_0).
	$$
	\begin{proposition}\label{energy-estimate}
		Assume that $s>\frac{(k-1)\alpha}{k}$, then we have for some $s_1<s$, close enough to $s$,
		\begin{equation}
		\begin{split}
		\sup_{0\leq t\leq T}E[w](t)\leq C_k(1+\|D^{s_1}z\|_{L^{\infty}([0,T]\times \T^d)}^{2k+2})\exp\left(T+\|D^{s_1}z\|_{L^{\infty}([0,T]\times \T^d)}^{2k+2} T\right).
		\end{split}
		\end{equation}
	\end{proposition}
 Note that  when $d=3,\alpha=1, k=2$, the restriction $s>\frac{1}{2}$ coincides with the one in \cite{OTh-Po}. Before giving a proof, 
 we briefly recall the idea of Oh-Pocovnicu, which will give us the restriction $s>\frac{(k-1)\alpha}{k}$ in Proposition~\ref{energy-estimate}.
After integrating by part in $t$, the worst term in the expression of $E[w](t)$ is
$$ \int_{0}^t\int_{\T^d}w^{2k+1}\partial_tz=\int_{0}^t\int_{\T^d}w^{2k+1}D^{\alpha}z. $$
We could then distribute $D^{\alpha-s}$ to the $w^{2k+1}$ side. Then in principle, we need estimate
\begin{equation}\label{**} \int_0^t\|w^{2k}(t)D^{\alpha-s}w(t)\|_{L^1(M)}dt.
\end{equation}
Now we put $D^{\alpha-s}w\in L^{k+1}(\T^d)$ and $w^{2k}\in L^{\frac{k+1}{k}}(\T^d)$. Then by the Gagliardo-Nirenberg inequality, we have
$$ \|D^{\alpha-s}w\|_{L^{k+1}}\leq C\|D^{\alpha}w\|_{L^2}^{\theta} \|w\|_{L^{2k+2}}^{1-\theta},
$$
with
$$ \frac{1}{k+1}=\frac{\alpha-s}{d}+\left(\frac{1}{2}-\frac{\alpha}{d}\right)\theta+\frac{1-\theta}{2(k+1)}.
$$
Thus we majorize \eqref{**} by 
$$ C_TE[w]^{\frac{k}{k+1}}\cdot E[w]^{\frac{\theta}{2}+\frac{1-\theta}{2k+2}}.
$$ 
To apply Grownwall, the restriction $\frac{\theta}{2}+\frac{1-\theta}{2k+2}\leq \frac{1}{k+1}$ must hold, hence (modulo the end-point issue)
$ s> \frac{k-1}{k}\alpha.
$

	\begin{proof}[Proof of Proposition \ref{energy-estimate}]
	We have that	$w$ satisfies the equation
		$$ \partial_t^2w+D^{2\alpha}w +(z+w)^{2k+1}=0. 
		$$
		Denote by the energy of $w$
		$$E[w](t)=\int_{\T^d}\frac{1}{2}\left(|\partial_tw|^2+|D^{\alpha}w|^2\right)+\int_{\T^d}\frac{1}{2k+2}w^{2k+2},
		$$
		and we have 
		\begin{equation} \frac{dE}{dt}=-\sum_{m=1}^{2k+2}\binom{2k+2}{m}\int_{\T^d}\frac{1}{2k+2-m}\partial_t(w^{2k+2-m})z^m.
		\end{equation}
		This yields
		\begin{equation}
		\begin{split}
		E[w](t)=&\sum_{m=1}^{2k+2}\frac{\binom{2k+2}{m}}{2k+2-m}\int_0^t\int_{\T^d} w^{2k+2-m}\partial_t (z^m)dxdt'\\
		-&\sum_{m=1}^{2k+2}\frac{\binom{2k+2}{m}}{2k+2-m}\int_{\T^d}z^m(t,\cdot)w^{2k+2-m}(t,\cdot) dx\\
		=:& \mathrm{I}+\mathrm{II}.
		\end{split}
		\end{equation}
		
		Each summation in the  second term II can be bounded by Young's inequality as:
		$$ \epsilon\int_{\T^d} |w(t)|^{2k+2}dx+C(\epsilon)\int_{\T^d}|z(t)|^{2k+2}dx\leq \epsilon E[w](t) +C(\epsilon)\|z(t)\|_{L^{2k+2}}^{2k+2}
		$$
		for any $\epsilon>0$.
		Thus
		$$ \mathrm{II}\leq C(\epsilon)\cdot C_{k}E[w](t)+C(\epsilon)C_k\|z(t)\|_{L^{2k+2}}^{2k+2}.
		$$
		Now we estimate $\mathrm{I}$. Noticing that $\partial_t (z^m)=mz^{m-1} \partial_t z$, the term $\int_{\T^d}w^{2k+2-m}\cdot \partial_t(z^m)$ is of the form
		$$ \int_{\T^d} w^{p+1-m}Z\cdot D^{\alpha} U
		$$
		with $Z=z^{m-1}, U=D^{-\alpha}\partial_t z$,
		and $p=2k+1$. Note that if $m>1$, for any $\sigma_1>\sigma_2$
		$$ \|D^{\sigma_2}Z\|_{L^{\infty}(\T^d)}\leq C_{\sigma_1,\sigma_2}\|D^{\sigma_1}z\|_{L^{\infty}(\T^d)}\|z\|_{L^{\infty}(\T^d)}^{m-2},
		$$ 
		and
		$$ \|D^{\sigma_1}U\|_{L^{\infty}(\T^d)}\leq C\|D^{\sigma_1}z\|_{L^{\infty}(\T^d)}.
		$$
		
		From Littlewood-Paley decomposition, we could write
		\begin{equation*}
			\begin{split}
				\int_{\T^d}w^{p+1-m}Z\cdot Z\cdot D^{\alpha} U=& \sum_{N,N': N'/2\leq N\leq 2N'} \int_{\T^d} P_N(w^{p+1-m}\cdot Z) \cdot P_N(D^{\alpha} Z)\\
				\leq & \sum_{N\geq 1} N^{\alpha-s}\|D^s U\|_{L^{\infty}(\T^d)}\|P_N(w^{p+1-m}Z)\|_{L^1(\T^d)}.
			\end{split}
		\end{equation*}
We further decompose 
		\begin{equation*}
			\begin{split}
				&P_N\left(w^{p+1-m}Z\right)\\:=&\sum_{N_1\geq N_2\geq \cdots N_{p+1-m}}\sum_{j_0=1}^{p+1-m} \sum_{M\leq N_{j_0}}P_{N}\left(\prod_{\nu=1}^{p+1-m}P_{N_{\nu}}w\cdot P_M Z\right)\\
				+&\sum_{N_1\geq N_2\geq \cdots N_{p+1-m}}\sum_{M\geq N_{1}}P_N\left(\prod_{\nu=1}^{p+1-m}P_{N_{\nu}}w\cdot P_M Z\right)\\
				=&:\mathrm{I'}+\mathrm{II'}.
			\end{split}
		\end{equation*}
		Applying H\"older inequality, we obtain
		\begin{equation}\label{*}
			\begin{split}
				\|\mathrm{I'}\|_{L^1(\T^d)}  \leq  \sum_{N_1\geq N_2\geq \cdots N_{p+1-m}; N_1\geq \frac{N}{2(p+1)}}\sum_{j_0=1}^{p+1-m}\sum_{M\leq N_{j_0}} &\|P_{N_1}w\|_{L^{\frac{p+1}{m+1}}(\T^d)}\|P_M Z\|_{L^{\infty}(\T^d)}\\ \times &\left\|\prod_{j=2}^{p+1-m} P_{N_{\nu}}w\right\|_{L^{\frac{p+1}{p-m}}(\T^d)}.
			\end{split}
		\end{equation}
	From interpolation and Bernstein, 
		$$\|P_{N_1}w\|_{L^{\frac{p+1}{m+1}}(\T^d)}\leq N_1^{-\frac{2m\alpha}{p-1}}\|w\|_{H^{\alpha}(\T^d)}^{\frac{2m}{p+1}}\|w\|_{L^{p+1}(\T^d)}^{\frac{p-1-2m}{p-1}},
		$$
		we have
		\begin{equation*}
		\begin{split}
			\eqref{*}	\leq &\sum_{N_1\geq N_2\geq \cdots N_{p+1-m}; N_1\geq \frac{N}{2(p+1)}}\sum_{j_0=1}^{p+1-m}\sum_{M\leq N_{j_0}} CN_1^{-\frac{2m\alpha}{p-1}}\|P_MZ\|_{L^{\infty}(\T^d)}\|w\|_{H^{\alpha}(\T^d)}^{\frac{2m}{p-1}}\|w\|_{L^{p+1}}^{p-m+\frac{p-1-2m}{p-1}}\\
			\leq &C_p E[w]\cdot N^{-\frac{2m\alpha}{p-1}}\log(N)\|Z\|_{L^{\infty}(\T^d)}.
		\end{split}
		\end{equation*}
		
		To ensure the convergence of the dyadic sum in $N$, we need 
		$$ s>\frac{(p-2m-1)\alpha}{p-1}.
		$$
		
		For the term $\mathrm{II'}$, in a similar way, we have
		\begin{equation*}
			\begin{split}
				\|\mathrm{II'}\|_{L^1(\T^d)}
				\leq & C_p N^{-s_0}\|D^{s_0}Z\|_{L^{\infty}(\T^d)}\log(N)E[w]^{\frac{p+1-m}{p+1}}.
			\end{split}
		\end{equation*}
		To ensure the summability in $N$, we need 
		$$ s>s_0>\frac{\alpha}{2}.
		$$
		
		Therefore, if 
		$$ s>\max\left\{\frac{(k-1)\alpha}{k}, \frac{\alpha}{2}\right\},
		$$
		we can find some $s_0<s_1<s$, close to $s$ such that
		\begin{equation}\label{energy1}
		\begin{split}
		&\sum_{m=1}^{2k+2}\frac{\binom{2k+2}{m}}{2k+2-m}\left|\int_0^t\int_{\T^d}w^{2k+2-m}\partial_t (z^m)dxdt'\right|\\
		\leq & C_k\left(1+\|D^{s_1}z\|_{L^{\infty}([0,T];L^{\infty}(\T^d))}^{2k+2}\right)\int_0^tE[w](t')dt'.
		\end{split}
		\end{equation}
		Note that if $k\geq 2$, we have automatically that $\frac{(2k-2)\alpha}{2k}\geq \frac{\alpha}{2}$.
		An extra argument for the case $k=1$ is much simpler, following from a direct use of the Sobolev inequality. The proof of Proposition \ref{energy-estimate} is now complete.
	\end{proof}
	
	The almost sure boundedness of the linear evolution part is guaranteed by the following lemma, see for example Proposition 2.7 in \cite{Sun-Xia}.
	\begin{lemme}
		For any $1\leq q\leq \infty, 2\leq r\leq\infty$ and $\epsilon>0$, there exist $C,c>0$ such that for any $T>0$, 
		$$ \mathbb{P}\left[\|S(t)(u_0^{\omega},v_0^{\omega})\|_{L_t^qL_x^r([0,T]\times\T^d)}>\lambda\right]\leq C\exp\left(-\frac{c\lambda^2}{\max\{1,T^2\}\|(u_0,v_0)\|_{\mathcal{H}^{\epsilon}(\T^d)}^2}\right).
		$$
	\end{lemme}
	Now the probabilistic estimate above, the local well-posedness result (analogue of Proposition \ref{local_theory}) and Proposition \ref{energy-estimate} yield the following almost almost sure global well-posedness and the convergence result.
	\begin{proposition}\label{a.a.s.GWP}
		Given $\frac{(k-1)\alpha}{k}<s<\alpha$, for any data $(u_0,v_0)\in \mathcal{H}^s(\T^d)$, let $(u_0^{\omega},v_0^{\omega})$ be the randomisation defined as \eqref{randomization-general}. Then given any $T>0, \epsilon>0$, there exists $\Omega_{T,\epsilon}\subset \Omega$ such that \begin{enumerate}
			\item $\mathbb{P}[\Omega\setminus\Omega_{T,\epsilon}]<\epsilon$.
			\item For any $\omega\in \Omega_{T,\epsilon}$, there exists a unique solution $(u^{\omega}(t),\partial_tu^{\omega}(t))$ to \eqref{FNLKG} with initial data $(u_0^{\omega}, v_0^{\omega})$ in the class:
			$$ (S(t)(u_0^{\omega},v_0^{\omega}),\partial_t S(t)(u_0^{\omega},v_0^{\omega}) )+ C([0,T];\mathcal{H}^{\alpha}).
			$$
			Moreover, the nonlinear part $w^{\omega}(t)=u^{\omega}(t)-S(t)(u_0^{\omega}, v_0^{\omega})$ satisfies the probabilistic energy bound:
			\begin{equation}\label{energybound}
			\sup_{0\leq t\leq T}\|(w^{\omega}(t),\partial_t w^{\omega}(t))\|_{\mathcal{H}^{\alpha}(M)}\leq C(T,\epsilon,\|(u_0,v_0)\|_{\mathcal{H}^s}).
			\end{equation}
			
			\item Denote by $(u_{0,N}^{\omega}, v_{0,N}^{\omega})=\Pi_N(u_0^{\omega}, v_0^{\omega})$, then for any $\omega \in \Omega_{T,\epsilon}$, the smooth solution $(u_N^{\omega}, \partial_tu_N^{\omega})$ of \eqref{FNLKG} with initial data $(u_{0,N}^{\omega}, v_{0,N}^{\omega})$ converges to the solution $(u^{\omega}(t), \partial_tu^{\omega}(t))$ constructed in (2), in $\mathcal{H}^s$.
		\end{enumerate}
	\end{proposition} 
	The proof of (1) and (2) in this proposition is standard, see for example \cite{OTh-Po} or \cite{Sun-Xia}. The proof of (3) follows from the similar argument as in Section 5. The key point is the analogue of Lemma \ref{convergence-shortime} which guarantees the convergence in a short time interval. Then thanks to the global energy bound \eqref{energybound} of the nonlinear part  $w(t)$, the time interval of the local convergence can be chosen to be uniform. Finally, we obtain the convergence up to time $t=T$.
	
	To pass to the global existence and convergence, we define the set 
	$$ \Omega_{T}:=\bigcup_{k=1}^{\infty}\Omega_{T,2^{-k}}.
	$$
	We have that $\Omega_T$ is of full probability. Now let $\displaystyle{\widetilde{\Sigma}:=\limsup_{T\rightarrow\infty}\Omega_T}$, then $\widetilde{\Sigma}$ still has full measure. Furthermore, for any $\omega\in \widetilde{\Sigma}$, the conclusions (2), (3) of 
	Proposition~\ref{a.a.s.GWP} hold true up to $T=+\infty$.

	


\begin{thebibliography}{10}
		
		\bibitem{A-Tz} A.~Ayache, N.~Tzvetkov, {\it $L^p$ properties for Gaussian random series}, Tran. Amer. Math. Soc. 360 (2008), 4425-4439.
		
		\bibitem{BCT-livre} H.~Bahouri, J.-Y.~Chemin,  R.~Danchin,  {\it Fourier analysis and nonlinear partial differential equations,}
		Grundlehren der Mathematischen Wissenschaften.  343 Springer, Heidelberg, 2011. 
		%
		
		\bibitem{B} J.~Bourgain, {\it Periodic nonlinear {S}chr{\"o}dinger equation and invariant measures}, Comm. Math. Phys. 166 (1994), 1-26.
		
		\bibitem{bourgain} J.~Bourgain, {\it Invariant measures for the 2d-defocusing nonlinear {S}chr{\"o}dinger equation}, Comm. Math. Phys. 176 (1996), 421-445.
		%
		
		\bibitem{BGT} N.~Burq, P.~G{\'e}rard, N.~Tzvetkov,  {\it Strichartz inequalities and the nonlinear {S}chr\"odinger equation on compact manifolds}, Amer. J. Math. 126 (2004), 569-605.
		%
		
		
		\bibitem{BL-AENS}
		N.~Burq, G.~Lebeau, {\it Injections de Sobolev probabilistes et applications},  Ann. Sci. Éc. Norm. Supér.  46 (2013),  917-962.
		
		
		\bibitem{BT1/2} N.~Burq,  N.~Tzvetkov, {\it Random data Cauchy theory for supercritical wave equations I. Local theory}, Invent. Math. 173 (2008), 449--475.
		
		\bibitem{BT2} N.~Burq, N.~Tzvetkov, {\it Random data Cauchy theory for supercritical wave equations II. A global existence result}, Invent. Math. 173 (2008), 477-496.
		
		\bibitem{BT-JEMS}  N.~Burq, N.~Tzvetkov, {\it  Probabilistic well-posedness for the cubic wave equation}, J. Eur. Math. Soc. 16 (2014), 1--30.
		%
		\bibitem{BTT-AIF} N.~Burq, L.~Thomann, N.~Tzvetkov, {\it Long time dynamics for the one dimensional non linear Schr\"odinger equation,}  Annales de l'Institut Fourier. 63 (2013), 2137--2198.
		%
		\bibitem{BTT-Toulouse} N.~Burq, L.~Thomann, N.~Tzvetkov, {\it Remarks on the Gibbs measures for nonlinear dispersive equations},  
		Ann. Fac. Sci. Toulouse Math. (6) 27 (2018),  527--597.
		%
		\bibitem{CCT} M.~Christ,  J.~Colliander, T.~Tao, {\it Ill-posedness for nonlinear Schr\"odinger and wave equations,}  Preprint, November~2003.
		%
		\bibitem{CKSTT} J.~Colliander,  M.~Keel, G.~Staffilani, H.~Takaoka, T.~Tao,  {\it Sharp global well-posedness for KdV and modified KdV on $\R$ and $\T$}, 
		J. Amer. Math. Soc. 16 (2003), 705--749.
		
			\bibitem{DKRV}
		F.~David, A.~Kupiainen,  R.~Rhodes, V.~Vargas, {\it  Liouville quantum gravity on the Riemann
			sphere,} Comm. Math. Phys. 342 (2016), 869--907.
		
		\bibitem{DRV}
		F.~David, R.~Rhodes, V.~Vargas, {\it  Liouville quantum gravity on complex tori,} J. Math. Phys. 342 (2016), 869-907.
		
			\bibitem{Ga}
		C.~Garban, {\it  Dynamical Liouville,} arXiv:1805.04507[math.P].
		
		\bibitem{GRV} C.~Guillarmou, R.~Rhodes, V.~Vargas, {\it Polyakov's formulation of $2d$  bosonic string theory,}   arXiv:1607.08467v3 [math-MP].
		%
		\bibitem{H} L.~H\"ormander, {\it The spectral function of an elliptic operator}, Acta Math. 121 (1968), 193-218.
		%
		\bibitem{IK} A.~Ionescu, C.~Kenig,  {\it Global well-posedness of the Benjamin-Ono equation in low-regularity spaces},  J. Amer. Math. Soc. 20 (2007), 753-798.
		
		\bibitem{KT} T.~Kappeler, P.~Topalov, {\it Global wellposedness of KdV in $H^{-1}(\T,\R)$},  Duke Math. J. 135 (2006), 327-360.		
		
		\bibitem{L} G.~Lebeau, {\it Perte de r\'egularit\'e pour les \'equation d'ondes sur-critiques}, Bull. Soc. Math. Fr. 133 (2005), 145-157.
		
		\bibitem{Oh_new} T.~Oh,  {\it  A remark on norm inflation with general initial data for the cubic nonlinear Schr\"odinger equations in negative Sobolev spaces},
		Funkcial. Ekvac. 60 (2017), 259-277.
		
		\bibitem{OTh} T.~Oh, L.~Thomann, {\it  Invariant Gibbs measures for the 2-$d$ defocusing nonlinear wave equations,}  to appear in Ann. Fac. Sci. Toulouse Math.
		
		\bibitem{OTh-Po}
		T.~Oh, O.~Pocovnicu, {\it  Probabilistic global well-posedness of the energy-critical defocusing quintic nonlinear wave
			equation on $\mathbb{R}^3$,} J. Math. Pures Appl. 105 (2016), 342-366.
		
		\bibitem{P-R-T}
		A.~Poiret, D.~Robert, L.~Thomann,{\it Probabilistic global well-posedness for the supercritical nonlinear harmonic oscillator,
		}
		Anal. PDE.
	7 (2014), 997-1026.
		
		
		\bibitem{Sun-Xia} C-M.~Sun, B.~Xia, {\it Probabilistic well-posedness for supercritical wave equations with periodic boundary condition on dimension three, } Illinois J.Math. 60 (2016), 481-503. 
		
		
		\bibitem{Tz-AIF} N.~Tzvetkov, {\it Invariant measures for the defocusing Nonlinear Schr\"odinger equation},  Annales de l'Institut Fourier. 58 (2008), 2543-2604.
		%
		\bibitem{Tz-Lecture}  N.~Tzvetkov,  {\it Random data wave equations},  arXiv:1704.01191 [math.AP].
		%
		\bibitem{xia} B.~Xia, {\it Equations aux d\'eriv\'ees partielles et al\'ea}, PhD thesis University of Paris Sud, July 2016. 
		
	\end{thebibliography}
\end{document}